\documentclass[12pt,a4paper,reqno]{amsart}
\usepackage{amsfonts,amsthm,amsmath,amssymb,%amscd,
amsbsy,euscript}%,mystyle} %,textcomp}

\newtheorem{theor}{Theorem}%[section]

\theoremstyle{definition}

\newtheorem{convention}{Convention}
\newtheorem{state}[theor]{Proposition}%[section]
\newtheorem{proposition}[theor]{Proposition}%[section]
\newtheorem{lemma}[theor]{Lemma}%[section]
\newtheorem{cor}[theor]{Corollary}%[section]
\newtheorem{define}{Definition}%[section]

\newtheorem{example}{Example}%[section]
\theoremstyle{remark}
\newtheorem{rem}{Remark}%[section]

\newcommand{\pinner}{\mathbin{\mathchoice
{\hbox{\vrule width0.6em depth0pt height0.4pt
	\vrule width0.4pt depth0pt height0.8ex}}
{\hbox{\vrule width0.6em depth0pt height0.4pt
	\vrule width0.4pt depth0pt height0.8ex}}
{\hbox{\kern0.14em
	\vrule width0.48em depth0pt height0.4pt
	\vrule width0.4pt depth0pt height0.6ex\kern0.14em}}
{\hbox{\kern0.1em
	\vrule width0.39em depth0pt height0.4pt
	\vrule width0.4pt depth0pt height0.5ex\kern0.1em}}}}

\let \wt=\widetilde
\newcommand{\BBR}{\mathbb{R}}\newcommand{\BBC}{\mathbb{C}}

\newcommand{\BBS}{\mathbb{S}}
\newcommand{\BBZ}{\mathbb{Z}}
\newcommand{\EuA}{{{\EuScript A}}}

\newcommand{\cA}{{{\EuScript A}}}%{\mathcal{A}}
\newcommand{\bcA}{\boldsymbol{\mathcal{A}}}

\newcommand{\bcP}{{\boldsymbol{\mathcal{P}}}}

\newcommand{\cE}{\mathcal{E}}

\newcommand{\cL}{\mathcal{L}}

\newcommand{\cP}{\mathcal{P}}

\newcommand{\bs}{{\boldsymbol{s}}}
\newcommand{\bolds}{{\boldsymbol{s}}}
\newcommand{\bu}{{\boldsymbol{u}}}
\newcommand{\bv}{{\boldsymbol{v}}}
\newcommand{\bw}{{\boldsymbol{w}}}
\newcommand{\bx}{{\boldsymbol{x}}}
\newcommand{\bby}{{\boldsymbol{y}}}
\newcommand{\bz}{{\boldsymbol{z}}}

\newcommand{\bxi}{{\boldsymbol{\xi}}}

\newcommand{\dd}{\partial}
\newcommand{\Id}{{\mathrm d}}

\newcommand{\bbx}{{\boldsymbol{x}}}

\DeclareMathOperator{\dvol}{d%\,
	vol}

\newcommand{\Sl}{\mathfrak{sl}}

\DeclareMathOperator{\Jac}{Jac}

\DeclareMathOperator{\supp}{supp}

\newcommand{\lshad}{[\![}
\newcommand{\rshad}{]\!]}
\newcommand{\ov}{\overline}

\newcommand{\KdV}{{\text{KdV}}}

\newcommand{\mKdV}{{\text{mKdV}}}

\newcommand{\Liou}{{\text{Liou}}}

\DeclareMathOperator{\Assoc}{Assoc}

\newcommand{\by}[1]{\textit{{#1}}}
\newcommand{\jour}[1]{\textit{{#1}}}
\newcommand{\vol}[1]{\textbf{{#1}}}
\newcommand{\book}[1]{\textrm{{#1}}}

\title[Deformation approach to quantisation of field models]{The deformation quantization mapping\\[2pt] of Poisson\/-\/{} to associative structures\\[2pt] in field theory}

\author[Arthemy Kiselev]{Arthemy V. Kiselev${}^{\text{$*$,$\S$}}$}
\thanks{\textit{Address}: Johann Ber\-nou\-lli Institute for Mathematics and Computer Science, University of Groningen,
P.O.~Box 407, 9700~AK Groningen, The Netherlands. 
\quad\textit{E-mail}: \texttt{A.V.Kiselev\symbol{"40}rug.nl}}
\thanks{${}^*$\:Max Planck Institute for Mathematics, Vi\-vats\-gas\-se~7, \mbox{D-53111} Bonn, Germany}
\thanks{${}^{\S}$\:\textit{Present address}: $\smash{\text{IH\'ES}}$, Le Bois\/--\/Marie, 35~route de Chartres, Bures\/-\/sur\/-\/Yvette, \mbox{F-91440} France}

\date{April 30, 2017}



\begin{document}
\begin{abstract}
Let $\{{\cdot},{\cdot}\}_{\bcP}$ be a variational Poisson bracket in a field model on an affine bundle~$\pi$ over an affine base manifold~$M^m$. Denote by~$\times$ the commutative associative multiplication in the Poisson algebra~$\bcA%\overline{\gM}^m(\pi)
$ of local functionals $\Gamma(\pi)\to\Bbbk$ that take field configurations to numbers.
By applying the techniques from geometry of iterated variations, %~\cite{gvbv}, 
we make well defined the deformation quantization map ${\times}\mapsto{\star}={\times}+\hbar\,\{{\cdot},{\cdot}\}_{\bcP}+\bar{o}(\hbar)$ that produces 
a noncommutative $\Bbbk[[\hbar]]$-\/linear star\/-\/product~$\star$ 
in% the Poisson algebra
~$\bcA%\overline{\gM}^m(\pi)[[\hbar]]
$.
%We inspect the combinatorics of the Kontsevich deformation quantization formula~${\times}\mapsto{\star}$ that worked in the finite dimension~\cite{KontsevichFormality} and --\,now in the field\/-\/theoretic set\/-\/up\,-- we verify the star\/-\/product associativity modulo~$\bar{o}(\hbar^4)$, also explaining why it could leak at high orders of the parameter~$\hbar$.
\end{abstract}
\maketitle

\subsection*{Introduction}
Starting from a Poisson bi\/-\/vector $\mathcal{P}$ on a given finite\/-\/dimensional affine Poisson manifold $(N^n,\mathcal{P})$, the Kontsevich graph summation formula~\cite{KontsevichFormality} yields an explicit deformation ${\times}\mapsto{\star_\hbar}$ of the commutative product~$\times$ in the algebra $\cA\mathrel{{:}{=}}C^\infty(N^n)$ of smooth functions. The new operation~$\star_\hbar$ on the space $\cA[[\hbar]]=C^\infty(N)[[\hbar]]$ of power series is specified by the Poisson structure on~$N$: namely, $f\mathop{{\star}_\hbar}g=f\times g+
%\text{const}\cdot
\hbar\,\{f,g\}_{\mathcal{P}}+\bar{o}(\hbar)$ such that all the bi\/-\/differential terms at higher powers of the formal parameter~$\hbar$ are completely determined by the Poisson bracket~$\{\,,\,\}_{\mathcal{P}}$ in the leading deformation term.
%(In the context of fields and strings, the constant is set to $\boldsymbol{i}/2$ so that the parameter~$\hbar$ is the usual Planck constant.) 
The deformed product~$\star_\hbar$ is no longer commutative if~$\mathcal{P}\neq0$ but it %always 
stays associative, 
\[
\Assoc_{\star_\hbar}(f,g,h)\mathrel{{:}{=}}
\bigl(f\mathop{{\star}_\hbar}g\bigr)\mathop{{\star}_\hbar}h - f\mathop{{\star}_\hbar}\bigl(g\mathop{{\star}_\hbar}h
\bigr) \doteq 0\qquad \text{for all }f,g,h\in C^\infty(N)[[\hbar]],
\]
by virtue of bi\/-\/vector's property $\Jac(\cP)\mathrel{{:}{=}} [\![\mathcal{P},\mathcal{P}]\!]=0$ to be Poisson, cf.~\cite{sqs15}.
%so, modulo $[\![\mathcal{P},\mathcal{P}]\!]=0$ for.

In this paper we extend the Poisson set\/-\/up and graph summation technique in the deformation ${\times}\mapsto{\star_\hbar}$ to the jet\/-\/space (super)\/geometry of $N^n$-\/valued fields $\phi\in\Gamma(\pi)$ over another, $(m>0)$-\/dimensional affine base manifold~$M^m$ in a given %their 
bundle~$\pi$ and secondly, of variational Poisson bi\/-\/vectors~$\boldsymbol{\mathcal{P}}$ that encode the Poisson brackets~$\{\,,\,\}_{\boldsymbol{\mathcal{P}}}$ on the space of local functionals taking $\Gamma(\pi)\to\Bbbk$. We explain why an extension of the Kontsevich graph technique~\cite{KontsevichFormality}
is possible and how it is done by using the geometry of iterated variations~\cite{gvbv,cycle16}.
For instance, we derive a variational analogue of the Moyal associative $\star$-\/product, $f\mathop{\star}g=(f)\,\exp\bigl(\overleftarrow{\partial_i}\cdot\hbar\mathcal{P}^{ij}\cdot\overrightarrow{\partial_j}\bigr)\,(g)$, in the case where the coefficients~$\mathcal{P}^{ij}$ of bi\/-\/vector~$\mathcal{P}$ are constant (hence the Jacobi identity $[\![\mathcal{P},\mathcal{P}]\!]=0$ holds trivially). 
%%%
%By using several well\/-\/known examples of variational Poisson bi\/-\/vectors~$\boldsymbol{\mathcal{P}}$, we illustrate the construction of each bi\/-\/differential term in~$\star_\hbar$ in the general case, i.e., for Hamiltonian total differential operators with coefficients depending on the fields~$\phi$ and their derivatives; we then verify that the noncommutative quantised product~$\star_\hbar$ is associative by virtue of $[\![\boldsymbol{\mathcal{P}},\boldsymbol{\mathcal{P}}]\!]=0$. 
To process variational Poisson structures with nonconstant coefficients, we analyse (see~\cite{sqs15,cpp}) the factorization mechanism $\Assoc_{\star_\hbar}(f,g,h)=\Diamond\bigl(\cP,\Jac_\cP(\cdot,\cdot\cdot)\bigr)(f,g,h)$ using the Kontsevich graphs at higher powers of the deformation parameter~$\hbar$. We explain why, holding up to~$\bar{o}(\hbar^2)$, the associativity of~$\star_\hbar$ can start leaking at orders~$\hbar^{\geqslant3}$ in the variational Poisson geometry of field\/-\/theoretic models.\\
\centerline{\rule{1in}{0.7pt}}

\noindent%
Its concept going back to Weyl\/--\/Gr\"onewold~\cite{Groenewold} and Moyal~\cite{Moyal}, the problem of associativity\/-\/preserving deformation quantisation ${\times}\mapsto{\star}_\hbar$ of commutative product~$\times$ in the algebras~$C^\infty(N^n)$ of functions on smooth finite\/-\/dimensional symplectic manifolds~$(N^n,\omega)$ was considered by Bayen\/--\/Flato\/--\/Fr\o{}ns\-dal\/--\/Lichnerowicz\/--\/Sternheimer~\cite{BFFLS}. Further progress within the symplectic picture was made by De~Wilde\/--\/Lecomte~\cite{DeWildeLecomte} and independently, Fedosov~\cite{Fedosov}. To tackle the deformation quantisation problem in the case of finite\/-\/dimensional affine Poisson geometries $\bigl(N^n,\{\cdot,\cdot\}_{\cP}\bigr)$ --\,that is, in absence of the Darboux lemma which guarantees the existence of canonical coordinates on a chart $U_\alpha\subseteq N^n$ in the symplectic case\,-- Kontsevich developed the graph complex technique~\cite{MKZurichICM,MKParisECM}; it yields an explicit construction of each term in the %perturbation 
series ${\times}\mapsto{\star}_\hbar$, see~\cite{KontsevichFormality,Ascona96}.
We recall this approach and analyse some of its features in section~\ref{SecFinite} below.
%\footnote{The functionality of Kontsevich's algorithm for deformation quantisation ${\times}\mapsto{\star}_\hbar$ relies on the Formality statement$\smash{{}^{\text{\cite{KontsevichFormality}}}}$, which in turn refers to universal facts about all associative algebras$\smash{{}^{\text{\cite{OperadsAndMotives,Tamarkin98}}}}$.}
%%%
Specifically, %in every given system of local coordinates $\bu=(u^1$,\ $\ldots$,\ $u^n)$ in a chart $U_\alpha\subseteq N^n$, 
the %known 
sum over a suitable set of weighted oriented graphs %(see \S\ref{SecGraphs} below) 
determines on~$N^n\ni\bu$ a star\/-\/product~${\star}_\hbar$ which (\textit{i}) contains a given Poisson bracket $\{\cdot,\cdot\}_{\cP}$ in the leading deformation term at~$\hbar^1$ and which (\textit{ii}) is associative modulo the Jacobi identity for~$\{\cdot,\cdot\}_{\cP}$,
\begin{equation}\label{EqDefDiamond}
\Assoc_{\star_\hbar}(f,g,h) =
%\bigl(f\mathbin{{\star}_\hbar}g\bigr)\mathbin{{\star}_\hbar}h -
%f\mathbin{{\star}_\hbar}\bigl(g\mathbin{{\star}_\hbar}h\bigr) =
\Diamond\,\bigl(\cP,\Jac_\cP(\cdot,\cdot,\cdot)
%\{\{f,g\}_{\cP},h\}_{\cP}+\text{c.\,p.}
\bigr)(f,g,h), \qquad f,g,h\in C^\infty(N^n)[[\hbar]],
\end{equation}
where
\[
\Jac_\cP(a,b,c)= \tfrac{1}{2}\sum_{\sigma\in S_3} (-)^\sigma \{\{\sigma(a),\sigma(b)\}_\cP,\sigma(c)\}_\cP,
\qquad a,b,c\in C^\infty(N^n).
\]
The construction of polydifferential operator~$\Diamond$ has been analysed up to order~$\bar{o}(\hbar^4)$ in~\cite{sqs15,cpp}.

A key distinction between associativity mechanisms for the Darboux\/-\/symplectic and Poisson cases is a possibility of the star\/-\/product self\/-\/action on non\/-\/constant coefficients~$P^{ij}(\bu)$ of the bracket~$\{\cdot,\cdot\}_{\cP}$. It is readily seen that whenever those coefficients are constant, the graph summation formula for~${\star}_\hbar$ then yields the Moyal star\/-\/product~\cite{Moyal},
\begin{equation}\label{EqMoyal}
{\star}%_\hbar
\Bigr|_{\bu=(u^1,\ldots,u^n)}=\exp\Bigl(\frac{\overleftarrow{\dd}}{\dd v^i}\Bigr|_{v^i=u^i}\cdot\hbar P^{ij}(\bu)\cdot \frac{\overrightarrow{\dd}}{\dd w^j}\Bigr|_{w^j=u^j}\Bigr).
\end{equation}
This formula's geometric extension to the infinite\/-\/dimensional space of $N^n$-\/valued %physical 
fields over a given $m$-\/dimensional affine manifold~$M^m$ will be obtained in \S\ref{SecMoyal}, see Eq.~\eqref{EqVarMoyal} on p.~\pageref{EqVarMoyal} below.%\marginpar{Add}

Valid in finite\/-\/dimensional set\/-\/up, the result of~\cite{KontsevichFormality} was %at once 
known to be not working in the infinite dimension. %For instance, 
It could not be applied %without caution 
to field\/-\/theoretic models, should one attempt to assemble such geometries via a limiting procedure by first taking infinitely many ``zero\/-\/dimensional field theories'' over the discrete topological space
$M^0=\bigcup_{\bx\in M^m}\{\bx\}$. In fact, not only is the geometry %concept 
of $N^n$-\/valued %physical 
fields (here, $n$~internal degrees of freedom attached at every base point~$\bx\in M^m$) infinite\/-\/dimensional if~$m>0$ but also does the mathematical apparatus to encode it become substantially more complex, cf.~\cite{gvbv,cycle16}. Many elements of differential calculus are %were equally well 
known to be fragile in the course of transition from finite\/-\/dimensional geometry of~$N^n$ to the infinite jet spaces $J^\infty(\pi)$ for the bundles~$\pi$ of $N^n$-\/valued fields over~$M^m$, or to the infinite jet spaces of maps $J^\infty(M^m\to N^n)$, cf.~\cite{TMPh2010%Galli10
} vs~\cite{Vaintrob} and~\cite{Norway} vs~\cite{AKZS} or contrast~\cite{gvbv} 
vs~\cite{ASSchwarz1993%SchwarzBV1993
}, \cite{cycle16} vs~\cite{KontsevichCyclic}, and~\cite{KontsevichFormality} vs this paper.

%\footnote{The geometric \textsl{correspondence principle} is an heuristic, subject to revision set of rules which track the various patterns and hint analogies between the algebraic structures on finite\/-\/dimensional Poisson manifolds $\bigl(N^n$, $\{\cdot,\cdot\}_{\cP}\bigr)$ and such structures' namesakes over the infinite jet spaces~$J^\infty(\pi)$, see~\cite{Topical11}. By making such correspondence work in the reverse direction, we shall not only construct the star\/-\/products~${\star}_\hbar$ of local functionals $\Gamma(\pi)\to\Bbbk$ defined at sections~$\phi$ of the bundles~$\pi$ but also 
%point out %reveal identify 
%their well\/-\/defined %natural 
%reductions in the classical set\/-\/up (see~\S\ref{SecScalar}).}

The aim of this paper is to develop %describe 
a tool for regular deformation quantisation $({\bcA},{\times})\mapsto\bigl({\bcA}[[\hbar]],{\star}_\hbar\bigr)$ of field theory models. The commutative associative unital algebras $({\bcA},{\times})$ of local functionals %(or \textsl{observables}) 
equipped with variational Poisson structures $\{\cdot,\cdot\}_{\bcP}$ are the input data of quantisation algorithm; 
in the output one obtains the noncommutative products~${\star}_\hbar$, associative up to~$\bar{o}(\hbar^{\geqslant2})$, in the unital algebras~${\bcA}[[\hbar]]$.

\smallskip
This paper is structured as follows. In the next section we 
review the %breadth of applicability for) Kontsevich's 
concept of deformation quantisation~% in its original phrasing
\cite{KontsevichFormality} for finite\/-\/dimensional affine Poisson manifolds~$\bigl(N^n$,\ $\{\cdot,\cdot\}_{\cP}\bigr)$: the deformation ${\times}\mapsto{\star}_\hbar$ is approached via summation over a class of weighted oriented graphs. 
In section~\ref{SecInfinite} we proceed %not just with manifolds~$N^n$ but 
with finite\/-\/dimensional affine bundles~$\pi$ 
%--\,in particular, affine %gauge
%bundles\,-- 
over affine manifolds~$M^m$; the infinitesimal parts of deformations are now specified by the variational Poisson brackets~$\{\cdot,\cdot\}_{\bcP}$ for algebras~$\bcA$ of local functionals taking~$\Gamma(\pi)\to\Bbbk$. To extend the deformation quantisation technique to such set\/-\/up, we let elements of the Gel'fand calculus of singular linear integral operators %and geometry of iterated variations 
enter the game in~\S\ref{SecElements}. Each graph in the Kontsevich summation formula now encodes a local variational (poly)\/differential operator. 
We then inspect in~\S\ref{SecJacobi} the geometric %analytic 
mechanism through which the new star\/-\/products can stay associative.
Taking star\/-\/product~\eqref{EqMoyal} by Weyl\/--\/Gr\"onewold\/--\/Moyal  as prototype, we illustrate the algorithm of variational deformation quantisation by presenting %in~\S\ref{SecMoyal} 
this structure's associative %proper 
analogue~${\star }$ for the class of Hamiltonian total differential operators $\bigl\|P^{ij}_\tau(\bx)\,\Id^{|\tau|}/\Id\bx^\tau\bigr\|_{i=1\ldots n}^{j=1\ldots n}$ %constant 
whose coefficients do not depend on sections%physical field portraits
~$\bu=\phi(\bx)$.
%In~\S\ref{SecWhyLeaks} 
On the other hand, for field\/-\/dependent Hamiltonian operators we indicate %here the 
a channel for the %its 
associativity to leak at orders~$\bar{o}(\hbar^{\geqslant 2})$. %or at higher degrees of~$\hbar$.
This effect was altogether suppressed in the seminal picture% of~\cite{KontsevichFormality}
; originally invisible, it can appear only in the framework of fibre bundles~$\pi$ over the base manifold~$M^m$ of positive dimension~$m$. 

\section{%\S2. 
The Kontsevich $\star_\hbar$-\/product %deformation quantisation 
on finite\/-\/dimensional %affine spaces
Poisson manifolds}\label{SecFinite}
\noindent%
In this section we recall the graph technique~\cite{KontsevichFormality} for deformation quantisation ${\times}\mapsto{\star}_{\hbar}$ on finite\/-\/dimensional affine Poisson manifolds~$\bigl(N^n$,\ $\{\cdot,\cdot\}_{\cP}\bigr)$.

\subsection{}\label{SecGraphs}%{\S~2.1.}
Let us first consider the direct problem of producing Lie algebra structures from a given associative product in the algebra of functions on~$N^n$.

\begin{lemma}
Let $\EuA$ be an associative algebra. 
Denote by~$\star$ the associative multiplication in~$\EuA$. Then the bi\/-\/linear skew\/-\/symmetric operation
\begin{equation}\label{EqLieBracket}
\{f,g\}\mathrel{{:}{=}} f\star g - g \star f,\qquad f,g\in\EuA,
\end{equation}
is a Lie bracket satisfying the Jacobi identity.\footnote{For example, let $\EuA$~be the algebra~$\cA=C^\infty(N^n)$ of smooth functions or the algebra~$\cA[[\hbar]]$ of formal power series on a given %smooth (in particular, affine) 
manifold~$N^n$. Then Lie bracket~\eqref{EqLieBracket} is not necessarily a bi\/-\/derivation and its differential order with respect to either of the arguments is not necessarily bounded.}
\end{lemma}

\begin{proof}
Indeed, the Jacobiator $\Jac_{\{\cdot,\cdot\}}(f,g,h) = \sum_{\sigma\in\circlearrowright}\{\{\sigma(f),\sigma(g)\},\sigma(h)\}$ is assembled by using the sum of associators:
\[
\Jac_{\{\cdot,\cdot\}}(f,g,h) = \sum_{\tau\in S_3} (-)^\tau\, \Assoc_\star\bigl(\sigma(f),\sigma(g),\sigma(h)\bigr),
\qquad f,g,h\in\EuA.
\]
This tells us that the Jacobi identity for the bracket~$\{\cdot,\cdot\}$ is an obstruction to the associativity of the product~$\star$.
\end{proof}

\begin{cor}
Let $\{\cdot,\cdot\}_\cP$ be a Poisson bracket on~$N^n$ and~$\times$ be the multiplication in the algebra~$\cA=C^\infty(N^n)$. Suppose that a deformation $\times\mapsto\star=\times+\hbar\,\{\cdot,\cdot\}_\cP+\bar{o}(\hbar)$ of the product in~$\cA$ to a multiplication in~$\cA[[\hbar]]$ is such that $\star$~is associative at all orders of the deformation parameter~$\hbar$. Then this deformation $\times\mapsto\star$ yields a transformation $\{\cdot,\cdot\}_\cP\mapsto \{\cdot,\cdot\}=\{\cdot,\cdot\}_\cP+\bar{o}(1)$ of the Poisson bracket $\{\cdot,\cdot\}_\cP$ to a Lie, but not necessarily Poisson bracket~\eqref{EqLieBracket}.
\end{cor}

%Let $\{\cdot,\cdot\}_\cP$ be a Poisson bracket on~$N^n$ and 
\begin{lemma}[\cite{KontsevichFormality}]
Denote by~$\times$ the associative multiplication in the algebra~$\cA=C^\infty(N^n)$. Suppose that a deformation $\times\mapsto\star=\times+\hbar\,B_1(\cdot,\cdot)+\hbar^2\,B_2(\cdot,\cdot)+\bar{o}(\hbar^2)$ is such that $B_1(\cdot,\cdot)$~is a %skew\/-\/symmetric 
bi\/-\/derivation and let $\star\ \text{mod}\:\bar{o}(\hbar^2)$ be associative up to~$\bar{o}(\hbar^2)$, that is,
\[
\Assoc_\star(f,g,h) = \bar{o}(\hbar^2)\qquad \text{for } f,g,h\in\cA[[\hbar]].
\]
Then the bi\/-\/linear skew\/-\/symmetric bi\/-\/derivation\footnote{By assumption, the leading deformation term $\hbar B_1(\cdot,\cdot)$ in $\star$ is a bi\/-\/derivation, hence same are its symmetric and skew\/-\/symmetric parts, $B_1^+(f,g)=\frac12(B_1(f,g)+B_1(g,f))$ and $B_1^-(f,g)=\frac12(B_1(f,g)-B_1(g,f))$, respectively.}
\[
\{f,g\}_\star\mathrel{{:}{=}}\left.\frac{f\star g-g\star f}{\hbar}\right|_{\hbar\mathrel{{:}{=}}0}
=2B_1^{-}(f,g)
\]
is a Poisson bracket.
\end{lemma} 

\begin{proof}
In the leading order, $\Assoc_{\star\ \text{mod}\:\bar{o}(\hbar^2)}(f,g,h)=\Assoc_{\times}(f,g,h)+\bar{o}(1)=\bar{o}(1)$.
At~$\hbar^1$ we have that $\Assoc_{\star\ \text{mod}\:\bar{o}(\hbar^2)}(f,g,h)=\hbar\,\bigl[ B_1(f,g)\times h + B_1(f\times g,h) - f\times B_1(g,h) - B_1(f,g\times h)\bigr]+\bar{o}(\hbar)=\bar{o}(\hbar)$ because $B_1$~is a derivation with respect to each argument. Next, at~$\hbar^2$ we obtain that
\begin{multline*}
\Assoc_{\star\ \text{mod}\:\bar{o}(\hbar^2)}(f,g,h) = \hbar^2\,\bigl[B_1(B_1(f,g),h) - B_1(f,B_1(g,h))\\
-\bigl(f\times B_2(g,h) - B_2(f,g)\times h + B_2(f,g\times h) - B_2(f\times g,h)\bigr)\bigr] + \bar{o}(\hbar^2).
\end{multline*}
We are given that this expression's leading term vanishes. By taking an alternating sum over the group of permutations of three arguments and recalling that %(\textit{i}) cyclic permutations are parity\/-\/even and (ii)
the product~$\times$ is commutative, we deduce that
\begin{multline*}
\sum_{\sigma\in S_3} (-)^\sigma\,\Assoc_{\star\ \text{mod}\:\bar{o}(\hbar^2)}\bigl(\sigma(f),\sigma(g),\sigma(h)\bigr) ={}\\
\sum_{\sigma\in S_3} (-)^\sigma\,\bigl[ B_1\bigl(B_1(\sigma(f),\sigma(g)),\sigma(h)\bigr)
-B_1\bigl(\sigma(f),B_1(\sigma(g),\sigma(h)\bigr) \bigr] + \bar{o}(\hbar^2),
\end{multline*}
that is, all the four terms containing~$B_2(\cdot,\cdot)$ cancel out; in the sum over permutations they are grouped by (1st~$-$ 2nd) $+$~(3rd) $-$~(4th). Finally, let us split $B_1(\cdot,\cdot)=B_1^{+}(\cdot,\cdot)+B_1^{-}(\cdot,\cdot)$ and obtain that in fact, its symmetric part also cancels out in the alternating sum:
\[
\sum_{\sigma\in S_3} (-)^\sigma\,\Assoc_{\star\ \text{mod}\:\bar{o}(\hbar^2)}\bigl(\sigma(f),\sigma(g),\sigma(h)\bigr) =
\Jac_{\{\cdot,\cdot\}_\star}(f,g,h)+\bar{o}(\hbar^2),
\]
whence the assertion readily follows.
\end{proof}

%%%%%%%%%%%%%%%%%%%%%%%%%%%%%%%%%%%%%% Example from GQT summer school (June 2015). %%%%%%%%%%%%%%%%%%%%%%%%%%%%
\begin{example}
Let $f,g$ be functions in the Cartesian coordinates~$p$ and~$q$ on~$\BBR^2$. Consider the associative star\/-\/product
\[
(f\star g)(p,q;\hbar)=f{\bigr|}_{(p,q)}\exp\Bigl(%\frac
{\overleftarrow{\dd}}/{\dd p}\cdot\hbar\cdot
%\frac
{\overrightarrow{\dd}}/{\dd q}\Bigr)g{\bigr|}_{(p,q)}.
\]
We have that
\[
\frac{\dd f}{\dd p}\cdot\frac{\dd g}{\dd q}=
\frac12\Bigl(\frac{\dd f}{\dd p}\cdot\frac{\dd g}{\dd q}+\frac{\dd g}{\dd p}\cdot\frac{\dd f}{\dd q}\Bigr)+
\frac12\Bigl(\frac{\dd f}{\dd p}\cdot\frac{\dd g}{\dd q}-\frac{\dd g}{\dd p}\cdot\frac{\dd f}{\dd q}\Bigr).
\]
By construction, we obtain
\[
\{f,g\}_{\star}=(f)%\frac
{\overleftarrow{\dd}}/{\dd p}\cdot%\frac
{\overrightarrow{\dd}}/{\dd q}(g)-
(g)%\frac
{\overleftarrow{\dd}}/{\dd p}\cdot%\frac
{\overrightarrow{\dd}}/{\dd q}(f),
\]
which is the two functions' Poisson bracket referred to the canonical Darboux coordinates~$p$ and~$q$.
\end{example}
%%%%%%%%%%%%%%%%%%%%%%%%%%%%%%%%%%%%%%%%%%%%%%%%%%%%%%%%%%%%%%%%%%%%%%%%%%%%%%%%%%%%%%%%%%%%%%%%%%%%%%%%%%%%%%%%

\begin{rem}
From now on we shall always assume %at once 
that the leading deformation term~$B_1(\cdot,\cdot)$ at~$\hbar^1$ in~$\star$ is skew\/-\/symmetric.
In the Kontsevich star\/-\/product, the symmetric part~$B_1^+$ of a given deformation term~$B_1$ might not be vanishing identically \emph{ab initio} but it can then be trivialised --~at the
expense of using suitable gauge transformations $f\mapsto f+\hbar D_1(f)+\ov{o}(\hbar)$, $g\mapsto g+\hbar D_1(g)+\ov{o}(\hbar)$ of its arguments (see~\cite{KontsevichFormality}).
\end{rem}
%Because the \textsl{symmetric} part $B_1^+(\cdot,\cdot)$ can always be gauged out from the star\/-\/product
%$f\star g=f\times g+\hbar B_1(f,g)+\ov{o}(\hbar)$,
%there would remain only the skew\/-\/symmetric term in the right\/-\/hand side of the formula

\subsection{}%2.2
Now suppose that a Poisson bracket $\{\cdot,\cdot\}_{\cP}$ is given on~$N^n$ in advance. 
Can the commutative associative multiplication~$\times$ in the algebra $C^{\infty}(N^n)\ni f,g$ be deformed to an associative %(possibly, modulo $\ov{o}(\hbar^2)$) 
star\/-\/product~$\star_{\hbar}$ such that the formal power series 
$f\mathbin{\star_{\hbar}}g=f\times g+\hbar\cdot\{f,g\}_{\cP}+\sum_{k=2}^{+\infty} \hbar^k\,B_k(f,g)$
is well defined% on the coordinate %affine 
%domain~$U_{\alpha}\subseteq\BBR^n$
\,? More specifically, the bi-linear, not necessarily commutative star-product $\star_{\hbar}=\times+\hbar\{\cdot,\cdot\}_{\cP}+\sum_{k>1}\hbar^k B_k(\cdot,\cdot)$
must satisfy the four axioms:
\begin{enumerate}
\item %(1) 
it is associative,
\begin{equation}\tag{\ref{EqDefDiamond}${}'$}\label{EqAssocModuloJacobi}
(f\mathbin{\star_{\hbar}}g)\mathbin{\star_{\hbar}}h\doteq
 f\mathbin{\star_{\hbar}}(g\mathbin{\star_{\hbar}}h)
\quad \text{via}\ \{\{f,g\}_{\cP},h\}_{\cP}+\text{c.\,p.}=0,\qquad
f,g,h\in C^{\infty}(N^n),
\end{equation}
i.e.\ modulo the property of bracket $\{\cdot,\cdot\}_{\cP}$ on~$N^n$ to be Poisson;%in~$C^{\infty}(U_{\alpha})$;
\item %(2) 
the unit function $1\in C^{\infty}(N^n)$ remains the neutral element for~$\star_{\hbar}$; whatever $f\in C^{\infty}(N^n)$, one has that
$f\mathbin{\star_{\hbar}}1=f=1\mathbin{\star_{\hbar}}f$;
\item %(3) 
each term $B_k(\cdot,\cdot)$, including the skew\/-\/symmetric Poisson bracket $\{\cdot,\cdot\}_{\cP}=B_1(\cdot,\cdot)$ to start with at~$\hbar$, is a bi\/-\/linear differential operator of bounded order;
\item %(4) 
the product $\star_{\hbar}$ is (let to be) $\Bbbk[[\hbar]]$-\/linear over~$C^{\infty}(N^n)[[\hbar]]$.
\end{enumerate}

\begin{theor}[\cite{KontsevichFormality}]
For every affine $n$-\/dimensional Poisson manifold $(N^{n<\infty},\cP)$ there exists a star\/-\/product $\star_\hbar=\times+\hbar\,\{\cdot,\cdot\}_\cP + \sum_{k=2}^{+\infty} \hbar^k\,B_k(\cdot,\cdot)$ in~$\cA[[\hbar]]$ satisfying the above four axioms.
\end{theor}

The proof is constructive (cf.~\cite{CattaneoFelderCMP2000} and~\cite{OperadsAndMotives,Tamarkin98});
the graph technique~\cite{KontsevichFormality,Ascona96,MKZurichICM,MKParisECM} 
is a convenient way to encode the bi\/-\/differential terms $B_k(\cdot,\cdot)$ in perturbation series~$\star_{\hbar}$.
Every term in~$B_k(f,g)$ at~$\hbar^k$, $k\geqslant0$ is encoded by an oriented graph~$\Gamma$ with $k+2$~vertices,
of which two sinks contain the respective arguments~$f$ and~$g$ and each of the remaining %/ other]
$k$~internal vertices is a source for two oriented edges. (In total, there are~$k$ such wedges with $2k$~arrows in every such graph~$\Gamma$.) 
Neither tadpoles nor multiple edges are permitted (cf.~\cite{CattaneoFelderCMP2000}). 
Next, install a copy of the given Poisson bi\/-\/vector~$\cP$ at each of the $k$~tops of the wedges 
and decorate every edge in the graph~$\Gamma$ at hand with a summation index running from~$1$ to~$n=\dim N^n$
%at every edge of the graph~$\Gamma$; the precedence\/-\/antecedence relation between the edges associates the indexes they carry with the respective indexes in bi-vector's copy.
The two edges issued from each internal vertex are \textsl{ordered}, so that the precedent and antecedent edges correspond to the first and second indexes in a copy of the Poisson bi-vector~$\cP$.

\subsection{}%2.3
To encode multi\/-\/vectors $\Xi\in\Gamma\bigl(\bigwedge^* TN^n\bigr)$ in a standard way, consider the parity-odd neighbour $\Pi T^*N^n$ of cotangent bundle to the manifold $N^n$ and denote by $\bxi=(\xi_1,\dots,\xi_n)$ the $n$-tuple of $\BBZ_2$-parity odd fibre coordinates over a chart $U_{\alpha}\subseteq N^n$
with an $n$-\/tuple $\bu=(u^1,\ldots,u^n)$ of local coordinates. 
Whenever the values $\{u^i,u^j\}_{\cP}%\bigr|_
{(\bu)}=P^{ij}(\bu)$ are given, % at $\bu=(u^1,\dots,u^n)\in U_{\alpha}$,
construct the bi-vector $\cP=\tfrac{1}{2}\langle\xi_i P^{ij}\bigr|_{\bu}\xi_j\rangle\in\Gamma(\bigwedge^2 TU_{\alpha})$;
bi-vectors are Poisson if they satisfy the classical master\/-\/equation $\lshad\cP,\cP\rshad=0$, see footnote~\ref{RefPageCME} on p.~\pageref{RefPageCME}.% below. %in what follows.

\begin{convention}
The correspondence between every decorated oriented %$(\xi\prec u)$
edge and analytic expressions occurring in $B_k(\cdot,\cdot)$ is %as follows 
established in Fig.~\ref{%Fig
EqContract};
\begin{figure}[htb]
%\begin{equation}\label{EqContract}
\[
\text{\raisebox{-37.5pt}{
\unitlength=1mm
\linethickness{0.4pt}
\begin{picture}(50.00,27.00)
\put(25.00,25.00){\vector(-1,-1){20.00}}
\put(3.33,3.33){\circle*{1.33}}
\put(26.33,26.33){\circle*{1.33}}
%\put(23.2,26.8){\line(1,-1){3.6}}
\put(13,15.33){\makebox(0,0)[cc]{$i$}}
\put(23.00,20){\makebox(0,0)[lc]{$\overrightarrow{\partial}\!/\partial\xi_i$}}
\put(10.33,7.33){\makebox(0,0)[lc]{$\overleftarrow{\partial}\!/\partial u^i$}}
\put(28.67,26.33){\makebox(0,0)[lc]{$\text{Obj}_{\text{tail}}$}}
\put(5.00,3.00){\makebox(0,0)[lc]{$\text{Obj}_{\text{head}}$}}
\put(32.00,16.00){\line(0,-1){2.67}}
\put(32.00,14.67){\vector(1,0){18.00}}
\end{picture}
}}
\sum_{i=1}^n(\text{Obj}_{\text{tail}})\frac{\overleftarrow{\partial}}{\partial\xi_i}\times\frac{\overrightarrow{\partial}}{\partial u^i}(\text{Obj}_{\text{head}})%.
\]
%\begin{equation}
\caption{The matching of indices in the derivative falling on the arrowhead object and in the Poisson bi\/-\/vector
stored in the arrowtail vertex is due to the coupling of the two objects' differentials taken with respect to the canonical conjugate variables.}\label{EqContract}
\end{figure}%{equation}
at every set of index values, the respective content of vertices in a connected graph component is multiplied using~$\times$ (cf.\ footnote~\ref{FootGraphIsFormula} on p.~\pageref{FootGraphIsFormula}).
The expressions determined by different connected components of one graph~$\Gamma$ in their formal sum are also multiplied by using the original product~$\times$.
\end{convention}

\begin{rem}
Because other arrows may stick into the vertices connected by an edge~% 
%$\stackrel{i}{\to}$ 
$\xrightarrow{\ i\ }$ in~Fig.~\ref{EqContract}, the objects
$\text{Obj}_{\text{tail}}$ and $\text{Obj}_{\text{head}}$ contained there can be derivatives (with respect to
$u^{\alpha}$'s) of the bi-vector~$\cP$ or, specifically to $\text{Obj}_{\text{head}}$ but never possible to
$\text{Obj}_{\text{tail}}$, arguments~$f$ and~$g$ of the star\/-\/product. On the same grounds, because there is another arrow issued from the tail vertex with $\text{Obj}_{\text{tail}}$, the formula encoded by a graph~$\Gamma$ does in fact not depend on any of the auxiliary, parity\/-\/odd variables~$\xi_j$. %on $\Pi T^*N^n$.
\end{rem}

In the Kontsevich star\/-\/product, every graph is accompained with its weight $w(\Gamma)\in\BBR$; these numbers are obtained by calculating certain explicitly
given integrals over the configuration spaces of $k$~distinct points --~in fact, the graphs' vertices containing~$\cP$~-- on the 
Lobachevsky plane (in its Poincar\'e model in the upper half-plane), see~\cite{KontsevichFormality}.
%%%
%Cattaneo and Felder$\smash{{}^{\text{\cite{CattaneoFelderCMP2000}}}}$ recall
%how the graphs (not only these; for tadpoles are also admissible) and their weights~$w(\Gamma)$ 
%arise %re-appear %/ show up
%through Feynman diagrams and path integral calculations in a known Poisson $\sigma$-\/model$\smash{{}^{\text{\cite{Ikeda1994,SchStr}}}}$.
The full set of rational values of weights for all graphs in an expansion $\star_\hbar\mod\bar{o}(\hbar^4)$ of the Kontsevich star\/-\/product has been obtained in~\cite{cpp}.

\begin{example}
For any %Hamiltonian 
functions $f,g\in C^{\infty}(N^n)$, the expansion %formal power series 
$f\star_{\hbar}g\mod\bar{o}(\hbar^2)$ %(on $U_{\alpha}$)
reads as follows:\footnote{The precedence\/-\/antecedence of edges is given %determinated here 
by the ordering of indexes $i\prec j$, $i_1\prec j_1$,
$i_2\prec j_2$, and $k\prec\ell$ in the analytic formula, see~\eqref{EqStarOh2} below.}
%%%
\begin{multline}\label{EqUniversalStar}
\text{\raisebox{-8.5pt}{
\unitlength=0.7mm
\linethickness{0.4pt}
\begin{picture}(12.67,5.67)
\put(2.00,5.00){\circle*{1.33}}
\put(12.00,5.00){\circle*{1.33}}
\put(7.00,5.00){\makebox(0,0)[cc]{$\star$}}
\put(2.00,1.33){\makebox(0,0)[cc]{$f$}}
\put(12.00,1.33){\makebox(0,0)[cc]{$g$}}
\end{picture}
}}
= %\\[-5mm]{=}
\text{\raisebox{-8.5pt}{
\unitlength=0.7mm
\linethickness{0.4pt}
\begin{picture}(15.00,5.67)
\put(0.00,5.00){\line(1,0){15.00}}
\put(2.00,5.00){\circle*{1.33}}
\put(13.00,5.00){\circle*{1.33}}
\put(2.00,1.33){\makebox(0,0)[cc]{$f$}}
\put(13.00,1.33){\makebox(0,0)[cc]{$g$}}
\end{picture}
}}
{+}\frac{\hbar^1}{1!}
\text{\raisebox{-12pt}{
\unitlength=0.7mm
\linethickness{0.4pt}
\begin{picture}(15.00,16.67)
\put(0.00,5.00){\line(1,0){15.00}}
\put(2.00,5.00){\circle*{1.33}}
\put(13.00,5.00){\circle*{1.33}}
\put(2.00,1.33){\makebox(0,0)[cc]{$f$}}
\put(13.00,1.33){\makebox(0,0)[cc]{$g$}}
\put(7.33,16.00){\circle*{1.33}}
\put(7.33,16.00){\vector(-1,-2){5.00}}
\put(7.33,16.00){\vector(1,-2){5.00}}
\end{picture}
}}
{+}\frac{\hbar^2}{2!}
\text{\raisebox{-12pt}{
\unitlength=0.7mm
\linethickness{0.4pt}
\begin{picture}(15.00,20.67)
\put(0.00,5.00){\line(1,0){15.00}}
\put(2.00,5.00){\circle*{1.33}}
\put(13.00,5.00){\circle*{1.33}}
\put(2.00,1.33){\makebox(0,0)[cc]{$f$}}
\put(13.00,1.33){\makebox(0,0)[cc]{$g$}}
\put(7.67,11.67){\circle*{1.33}}
\put(7.67,20.00){\circle*{1.33}}
\put(7.67,20.00){\vector(-1,-3){4.67}}
\put(7.67,20.00){\vector(1,-3){4.67}}
\put(7.67,11.67){\vector(-3,-4){4.00}}
\put(7.67,11.67){\vector(3,-4){4.33}}
\end{picture}
}}
{+}\frac{\hbar^2}{3}{\Biggl(}
\text{\raisebox{-12pt}{
\unitlength=0.7mm
\linethickness{0.4pt}
\begin{picture}(15.00,17.67)
\put(0.00,5.00){\line(1,0){15.00}}
\put(2.00,5.00){\circle*{1.33}}
\put(13.00,5.00){\circle*{1.33}}
\put(2.00,1.33){\makebox(0,0)[cc]{$f$}}
\put(13.00,1.33){\makebox(0,0)[cc]{$g$}}
\put(7.33,11.33){\circle*{1.33}}
\put(2.00,17.00){\circle*{1.33}}
\put(2.00,17.00){\vector(0,-1){11.33}}
\put(2.00,17.00){\vector(1,-1){5.33}}
\put(7.33,11.33){\vector(1,-1){5.33}}
\put(7.33,11.33){\vector(-1,-1){5.33}}
\end{picture}
}}
{-}%{+}
\text{\raisebox{-12pt}{
\unitlength=0.7mm
\linethickness{0.4pt}
\begin{picture}(15.00,18.00)
\put(0.00,5.00){\line(1,0){15.00}}
\put(2.00,5.00){\circle*{1.33}}
\put(13.00,5.00){\circle*{1.33}}
\put(2.00,1.33){\makebox(0,0)[cc]{$f$}}
\put(13.00,1.33){\makebox(0,0)[cc]{$g$}}
\put(7.33,11.33){\circle*{1.33}}
\put(7.33,11.33){\vector(1,-1){5.33}}
\put(7.33,11.33){\vector(-1,-1){5.33}}
\put(13.00,17.33){\circle*{1.33}}
\put(13.00,17.33){\vector(0,-1){11.67}}
\put(13.00,17.33){\vector(-1,-1){5.33}}
\end{picture}
}}
{\Biggr)}+\frac{\hbar^2}{6}
\text{\raisebox{-12pt}{
\unitlength=0.7mm
\linethickness{0.4pt}
\begin{picture}(15.00,20.33)
\put(0.00,5.00){\line(1,0){15.00}}
\put(2.00,5.00){\circle*{1.33}}
\put(13.00,5.00){\circle*{1.33}}
\put(2.00,1.33){\makebox(0,0)[cc]{$f$}}
\put(13.00,1.33){\makebox(0,0)[cc]{$g$}}
\put(2.00,15.00){\circle*{1.33}}
\put(13.00,15.00){\circle*{1.33}}
\put(13.00,15.00){\vector(0,-1){9.33}}
\put(2.00,15.00){\vector(0,-1){9.33}}
\bezier{64}(2.00,15.00)(7.00,9.00)(12.67,15.00)
\bezier{60}(13.00,15.00)(7.00,20.33)(2.67,15.00)
\put(11.67,14.00){\vector(1,1){0.67}}
\put(3.33,16.00){\vector(-1,-1){0.67}}
\end{picture}
}}
{+}\\
{}+\bar{o}(\hbar^2).
\end{multline}
Referred to any system of affine local coordinates $\bu=(u^1,\dots,u^n)$ on $U_{\alpha}\subseteq N^n$, 
for a given Poisson bi\/-\/vector $\cP{\bigr|}_{\bu}=\tfrac{1}{2}\langle\xi_iP^{ij}(\bu)\xi_j\rangle$ 
this sum of weighted graphs is realised by the formula%
\footnote{\label{FootGraphIsFormula}%
Note that a graph itself suggests the easiest\/-\/to\/-\/read way to write down the respective differential operator's formula; this %inscription of derivatives along the edges 
will be particularly convenient in the variational setting of section~\ref{SecInfinite}, see 
%Fig.~\ref{FigArrowVariations} on p.~\pageref{FigArrowVariations}.
p.~\pageref{pExampleGraphFormula}.}
%%%
\begin{multline}
f\mathbin{\star_{\hbar}}g=%)(\bu)=
f\times g+\frac1{1!}(f)\frac{\overleftarrow{\dd}}{\dd u^i}\cdot\hbar P^{ij}\cdot\frac{\overrightarrow{\dd}}{\dd u^j}(g)+
\frac1{2!}(f)
\left[
\begin{matrix}
\frac{\overleftarrow{\dd}}{\dd u^{i_1}_{\mathstrut}}\cdot\hbar P^{i_1j_1}\cdot\frac{\overrightarrow{\dd}}{\dd u^{j_1}}\\
\frac{\overleftarrow{\dd}}{\dd u^{i_2}}\cdot\hbar P^{i_2j_2}\cdot\frac{\overrightarrow{\dd}}{\dd u^{j_2}}
\end{matrix}
\right]
(g)+\\+
\frac1{3}\Biggl\{(f)\frac{\overleftarrow{\dd}}{\dd u^i}\frac{\overleftarrow{\dd}}{\dd u^k}\cdot
\hbar P^{ij}\cdot\frac{\overrightarrow{\dd}}{\dd u^j}(\hbar P^{k\ell})\cdot\frac{\overrightarrow{\dd}}{\dd u^{\ell}}(g)%-\\
-(f)\frac{\overleftarrow{\dd}}{\dd u^k}\cdot(\hbar P^{k\ell})\frac{\overleftarrow{\dd}}{\dd u^j}\cdot
\hbar P^{ij}\cdot\frac{\overrightarrow{\dd}}{\dd u^i}(g)\Biggr\}+{}\\
{}+\frac{1}{6}(f)\frac{\overleftarrow{\dd}}{\dd u^i}\cdot (\hbar P^{ij})\frac{\overleftarrow{\dd}}{\dd u^k}\cdot
\frac{\overrightarrow{\dd}}{\dd u^j}(\hbar P^{k\ell})\cdot \frac{\overrightarrow{\dd}}{\dd u^\ell}(g)
+\ov{o}(\hbar^2).\label{EqStarOh2}
\end{multline}
The values of (derivatives of) both arguments %Hamiltonian functions 
and coefficients of the Poisson bi\/-\/vector~$\cP$ are calculated
at~$\bu\in U_{\alpha}\subseteq N^n$ in the right\/-\/hand side of the above formula.%equality.
\end{example}

\begin{example}[Moyal\/--\/Weyl\/--\/Gr\"onewold]\label{ExMoyalAssoc}
Suppose that all coefficients~$P^{ij}$ of the Poisson bi\/-\/vector~$\cP$ are constant, which is a well defined property with respect to all local coordinate systems on the affine manifold~$N^n$ at hand.
In effect, the graphs with at least one arrow ariving at %/ coming to 
a vertex containing~$\cP$ make no contribution to the star\/-\/product~$\star$. The only %[surviving / 
contributing %/ meaningful]
graphs are portrayed in this figure,
\begin{equation*}
f\mathbin{\star}g=
\text{\raisebox{-12pt}{
\unitlength=1mm
\linethickness{0.4pt}
\begin{picture}(10.00,5.67)
\put(0.00,5.00){\line(1,0){10.00}}
\put(2.00,5.00){\circle*{1.33}}
\put(8.00,5.00){\circle*{1.33}}
\put(2.00,3.33){\makebox(0,0)[ct]{$f$}}
\put(8.00,3.00){\makebox(0,0)[ct]{$g$}}
\end{picture}
}}
+\frac{\hbar^1}{1!}
\text{\raisebox{-17.5pt}{
\unitlength=1mm
\linethickness{0.4pt}
\begin{picture}(10.00,14.00)
\put(0.00,5.00){\line(1,0){10.00}}
\put(8.00,5.00){\circle*{1.33}}
\put(2.00,5.00){\circle*{1.33}}
\put(5.00,13.00){\vector(-1,-3){2.33}}
\put(5.00,13.00){\vector(1,-3){2.33}}
\put(5.00,13.00){\circle*{1.33}}
\put(6.00,14.00){\makebox(0,0)[lb]{$\cP$}}
\put(2.00,3.33){\makebox(0,0)[ct]{$f$}}
\put(8.00,3.00){\makebox(0,0)[ct]{$g$}}
\end{picture}
}}
+\frac{\hbar^2}{2!}
\text{\raisebox{-22.5pt}{
\unitlength=1.00mm
\linethickness{0.4pt}
\begin{picture}(10.00,21.67)
\put(0.00,5.00){\line(1,0){10.00}}
\put(8.00,5.00){\circle*{1.33}}
\put(2.00,5.00){\circle*{1.33}}
\put(5.00,13.00){\vector(-1,-3){2.33}}
\put(5.00,13.00){\vector(1,-3){2.33}}
\put(5.00,13.00){\circle*{1.33}}
\put(2.00,3.33){\makebox(0,0)[ct]{$f$}}
\put(8.00,3.00){\makebox(0,0)[ct]{$g$}}
\put(5.00,21.00){\vector(-1,-4){3.67}}
\put(5.00,21.00){\vector(1,-4){3.67}}
\put(5.00,21.00){\circle*{1.33}}
\end{picture}
}}
+\frac{\hbar^3}{3!}
\text{\raisebox{-22.5pt}{
\unitlength=1.00mm
\special{em:linewidth 0.4pt}
\linethickness{0.4pt}
\begin{picture}(10.00,21.66)
\put(0.00,5.00){\line(1,0){10.00}}
\put(8.00,5.00){\circle*{1.33}}
\put(2.00,5.00){\circle*{1.33}}
\put(5.00,13.00){\vector(-1,-3){2.33}}
\put(5.00,13.00){\vector(1,-3){2.33}}
\put(5.00,13.00){\circle*{1.33}}
\put(2.00,3.33){\makebox(0,0)[ct]{$f$}}
\put(8.00,3.00){\makebox(0,0)[ct]{$g$}}
\put(5.00,21.00){\vector(-1,-4){3.67}}
\put(5.00,21.00){\vector(1,-4){3.67}}
\put(5.00,21.00){\circle*{1.33}}
\put(5.00,16.00){\vector(-1,-3){3.00}}
\put(5.00,16.00){\vector(1,-3){3.00}}
\put(5.00,16.33){\circle*{1.33}}
\end{picture}
}}
+\cdots+\frac{\hbar^k}{k!}
\text{\raisebox{-22.5pt}{
\unitlength=1.00mm
\special{em:linewidth 0.4pt}
\linethickness{0.4pt}
\begin{picture}(10.00,21.66)
\put(0.00,5.00){\line(1,0){10.00}}
\put(8.00,5.00){\circle*{1.33}}
\put(2.00,5.00){\circle*{1.33}}
\put(5.00,13.00){\vector(-1,-3){2.33}}
\put(5.00,13.00){\vector(1,-3){2.33}}
\put(5.00,13.00){\circle*{1.33}}
\put(5.00,17.00){\makebox(0,0)[cc]{$\vdots$}}
\put(2.00,3.33){\makebox(0,0)[ct]{$f$}}
\put(8.00,3.00){\makebox(0,0)[ct]{$g$}}
\put(5.00,21.00){\vector(-1,-4){3.67}}
\put(5.00,21.00){\vector(1,-4){3.67}}
\put(5.00,21.00){\circle*{1.33}}
\put(8.33,17.00){\makebox(0,0)[lc]{$\Biggr\}k$}}
\end{picture}
}}
+\cdots.
\end{equation*}
These graphs are such that their weights in the power series combine it to the Moyal exponent,
\begin{equation}\tag{\ref{EqMoyal}}
(f\mathbin{\star}g)(\bu;\hbar)=\left.\left[(f(\bu))\exp\left(\frac{\overleftarrow{\dd}}{\dd u^i}\cdot\hbar P^{ij}\cdot
\frac{\overrightarrow{\dd}}{\dd v^j}\right)(g(\bv))\right]\right|_{\bu=\bv}.
\end{equation}
Here we accept that the use of every next copy of the bi-vector~$\cP$ creates a new pair of summation indexes.
\end{example}

\begin{rem}
The introduction of two identical copies, $\bu\in U_{\alpha}\subseteq N^n$ and $\bv\in U_{\alpha}\subseteq N^n$, of the geometry where the objects~%Hamiltonians 
$f$ and~$g$ are defined %now
reveals %[is a forerunner of / hints to / heralds] the technique / 
an idea that will be used heavily in what follows.
\end{rem} 

\begin{state}\label{PropMoyalAssoc}
The associativity of Moyal star\/-\/product~\eqref{EqMoyal} is established %revealed 
by the \emph{a posteriori} congruence mechanism. 
\end{state}

\begin{proof}[{Proof \textup{(}\textup{see}~\textup{\cite{IndelicatoCattaneoPQR2003})}}]%Indeed, 
From the identity
\[
\left(f(u)\times g(v){\bigr|}_{u=v}\right)\frac{\overleftarrow{\dd}}{\dd u}=\left.
\left[(f(u)\times g(v))\left(\frac{\overleftarrow{\dd}}{\dd u}+\frac{\overleftarrow{\dd}}{\dd v}\right)\right]\right|_{u=v}
\]
we infer that $\left((f\star g)\star h-f\star(g\star h)\right)(\bu;\hbar)={}$%\\
\begin{multline*}
{}=\Bigl[\Bigl[(f|_{\bu})
\exp\Bigl(\frac{\overleftarrow{\dd}}{\dd u^i}\hbar P^{ij}\frac{\overrightarrow{\dd}}{\dd v^j}\Bigr)(g|_{\bv})\Bigr]
\Bigr|_{\bu=\bv}
\exp\Bigl(\frac{\overleftarrow{\dd}}{\dd u^k}\hbar P^{k\ell}\frac{\overrightarrow{\dd}}{\dd w^{\ell}}\Bigr)(h|_{\bw})\Bigr]\Bigr|_{\bu=\bw}-\\-
\Bigl[(f|_{\bu})
\exp\Bigl(\frac{\overleftarrow{\dd}}{\dd u^i}\hbar P^{ij}\frac{\overrightarrow{\dd}}{\dd v^j}\Bigr)\Bigl[(g|_{\bv})
\exp\Bigl(\frac{\overleftarrow{\dd}}{\dd v^k}\hbar P^{k\ell}\frac{\overrightarrow{\dd}}{\dd w^{\ell}}\Bigr)(h|_{\bw})\Bigr]\Bigr|_{\bv=\bw}\Bigr]\Bigr|_{\bu=\bv}=\\=
\Bigl[(f|_{\bu})
\exp\Bigl(\frac{\overleftarrow{\dd}}{\dd u^i}\hbar P^{ij}\frac{\overrightarrow{\dd}}{\dd v^j}\Bigr)(g|_{\bv})
\exp\Bigl(\Bigl(\frac{\overleftarrow{\dd}}{\dd u^k}+\frac{\overleftarrow{\dd}}{\dd v^k}\Bigr)\cdot
\hbar P^{k\ell}\frac{\overrightarrow{\dd}}{\dd w^{\ell}}\Bigr)(h|_{\bw})\Bigr]\Bigr|_{\bu=\bv=\bw}-\\-
\Bigl[(f|_{\bu})
\exp\Bigl(\frac{\overleftarrow{\dd}}{\dd u^i}\hbar P^{ij}\cdot
\Bigl(\frac{\overrightarrow{\dd}}{\dd v^j}+\frac{\overrightarrow{\dd}}{\dd w^j}\Bigr)
\Bigr)\Bigl[(g|_{\bv})
\exp\Bigl(\frac{\overleftarrow{\dd}}{\dd v^k}\hbar P^{k\ell}\frac{\overrightarrow{\dd}}{\dd w^{\ell}}\Bigr)(h|_{\bw})\Bigr]\Bigr|_{\bu=\bv=\bw}\equiv0,
\end{multline*}
which is due to the Baker\/--\/Campbell\/--\/Hausdorff formula for the exponent of sums of \textsl{commuting} derivatives,
and by having indexes relabelled.
\end{proof}

\subsection{}\label{SecJacVanishVia}%2.4
Whenever the coefficients~$P^{ij}(\bu)$ are not constant on the domain $U_{\alpha}\subseteq N^n$, the classical
master\/-\/equation\footnote{\label{RefPageCME}%
The Jacobi identity for Poisson bracket $\{\cdot,\cdot\}_{\cP}$ is equivalent to the zero\/-\/value condition
$\lshad\cP,\cP\rshad(f,g,h)=0$ for all Hamiltonians~$f,g,h$; the tri\/-\/vector $\lshad\cP,\cP\rshad$ is viewed here as a tri\/-\/linear totally antisymmetric mapping and we denote by $\lshad\cdot,\cdot\rshad$ the \textsl{Schouten bracket} (i.e., parity\/-\/odd Poisson bracket); 
in coordinates, one proves that
$\lshad\cP,\cP\rshad\stackrel{\text{Th.}}{=}
(\cP)\frac{\overleftarrow{\dd}}{\dd u^i}\cdot\frac{\overrightarrow{\dd}}{\dd\xi_i}(\cP)-
(\cP)\frac{\overleftarrow{\dd}}{\dd\xi_i}\cdot\frac{\overrightarrow{\dd}}{\dd u^i}(\cP).$}
%%%
$\lshad\cP,\cP\rshad=0$ is a nontrivial constraint for the bi\/-\/vector~$\cP$.
Where is the Jacobi identity for the Poisson bracket~$\{\cdot,\cdot\}_{\cP}$ hidden in the associator 
$(f\mathbin{\star_{\hbar}}g)\mathbin{\star_{\hbar}}h-f\mathbin{\star_{\hbar}}(g\mathbin{\star_{\hbar}}h)$ for the full star\/-\/product\,?

\begin{example}\label{ExFactorH2}
It is easy to see that $\Assoc_{\star_\hbar}(f,g,h)=\tfrac{2}{3}\Jac_\cP(f,g,h)+\bar{o}(\hbar^2)$.
\end{example}

By definition, we put
\begin{equation}\label{EqJacFig}
%\!\!\!\!\!\!\!\!
%\ \ 
%\vcenteredhbox{
\raisebox{3.3mm}%{13mm}
[6.5mm][3.5mm]{ % RB: how do I use less space?
\unitlength=1mm
\special{em:linewidth 0.4pt}
\linethickness{0.4pt}
\begin{picture}(12,15)
\put(0,-10){
\begin{picture}(12.00,15.00)
\put(0.00,10.00){\framebox(12.00,5.00)[cc]{$\bullet\ \bullet$}}
\put(2.00,10.00){\vector(-1,-3){1.33}}
\put(6.00,10.00){\vector(0,-1){4.00}}
\put(10.00,10.00){\vector(1,-3){1.33}}
\put(0.00,4.00){\makebox(0,0)[cb]{\tiny\it1}}
\put(6.00,4.00){\makebox(0,0)[cb]{\tiny\it2}}
\put(11.67,4.00){\makebox(0,0)[cb]{\tiny\it3}}
\end{picture}
}\end{picture}}%}
\ \ \ 
%%%
\mathrel{{:}{=}}
\text{\raisebox{-12pt}[25pt]{
\unitlength=0.70mm
\linethickness{0.4pt}
\begin{picture}(26.00,16.33)
\put(0.00,5.00){\line(1,0){26.00}}
\put(2.00,5.00){\circle*{1.33}}
\put(13.00,5.00){\circle*{1.33}}
\put(24.00,5.00){\circle*{1.33}}
\put(2.00,1.33){\makebox(0,0)[cc]{\tiny\it1}}
\put(13.00,1.33){\makebox(0,0)[cc]{\tiny\it2}}
\put(24.00,1.33){\makebox(0,0)[cc]{\tiny\it3}}
\put(7.33,11.33){\circle*{1.33}}
\put(7.33,11.33){\vector(1,-1){5.5}}
\put(7.33,11.33){\vector(-1,-1){5.5}}
\put(13,17){\circle*{1.33}}
\put(13,17){\vector(1,-1){11.2}}
\put(13,17){\vector(-1,-1){5.1}}
\put(3.00,10.00){\makebox(0,0)[cc]{\tiny$i$}}
\put(12.00,10.00){\makebox(0,0)[cc]{\tiny$j$}}
\put(24.00,10.00){\makebox(0,0)[cc]{\tiny$k$}}
\end{picture}
}}
{-}
\text{\raisebox{-12pt}[25pt]{
\unitlength=0.70mm
\linethickness{0.4pt}
\begin{picture}(26.00,16.33)
\put(0.00,5.00){\line(1,0){26.00}}
\put(2.00,5.00){\circle*{1.33}}
\put(13.00,5.00){\circle*{1.33}}
\put(24.00,5.00){\circle*{1.33}}
\put(2.00,1.33){\makebox(0,0)[cc]{\tiny\it1}}
\put(13.00,1.33){\makebox(0,0)[cc]{\tiny\it2}}
\put(24.00,1.33){\makebox(0,0)[cc]{\tiny\it3}}
\put(13,11.33){\circle*{1.33}}
\put(13,11.33){\vector(2,-1){10.8}}
\put(13,11.33){\vector(-2,-1){10.8}}
\put(18.5,17){\circle*{1.33}}
\put(18.5,17){\vector(-1,-1){5.2}}
\put(18.5,17){\vector(-1,-2){5.6}}
\put(13,15){\tiny $L$}
\put(17,12){\tiny $R$}
\put(4.00,10.00){\makebox(0,0)[cc]{\tiny$i$}}
\put(11.00,8.00){\makebox(0,0)[cc]{\tiny$j$}}
\put(22.00,10.00){\makebox(0,0)[cc]{\tiny$k$}}
\end{picture}
}}
{-}
\text{\raisebox{-12pt}[25pt]{
\unitlength=0.70mm
\linethickness{0.4pt}
\begin{picture}(26.00,16.33)
\put(0.00,5.00){\line(1,0){26.00}}
\put(2.00,5.00){\circle*{1.33}}
\put(13.00,5.00){\circle*{1.33}}
\put(24.00,5.00){\circle*{1.33}}
\put(2.00,1.33){\makebox(0,0)[cc]{\tiny\it1}}
\put(13.00,1.33){\makebox(0,0)[cc]{\tiny\it2}}
\put(24.00,1.33){\makebox(0,0)[cc]{\tiny\it3}}
\put(18.33,11.33){\circle*{1.33}}
\put(18.33,11.33){\vector(1,-1){5.5}}
\put(18.33,11.33){\vector(-1,-1){5.5}}
\put(13,17){\circle*{1.33}}
\put(13,17){\vector(-1,-1){11.2}}
\put(13,17){\vector(1,-1){5.1}}
\put(3.00,10.00){\makebox(0,0)[cc]{\tiny$i$}}
\put(13.00,10.00){\makebox(0,0)[cc]{\tiny$j$}}
\put(24.00,10.00){\makebox(0,0)[cc]{\tiny$k$}}
\end{picture}
}}
= 0. %\nonumber
\end{equation}%
\vskip .5em
\noindent
In formulae, by ascribing the index~$\ell$ to the unlabeled edge, the identity reads
\[
(
\partial_\ell \cP^{ij} \cP^{\ell k} +
\partial_\ell \cP^{jk} \cP^{\ell i} +
\partial_\ell \cP^{ki} \cP^{\ell j}
)\,
\partial_i({%\tiny
\it 1})\, \partial_j({%\tiny
\it 2})\, \partial_k({%\tiny
\it 3}) = 0.
\]
The coefficient of $\partial_i \otimes \partial_j \otimes \partial_k$ is the familiar form of the Jacobi identity.

To understand how sums of graphs can vanish by virtue of differential consequences of Jacobi identity~\eqref{EqJacFig}, let us note that for a given Poisson bi\/-\/vector~$\cP$ and for every derivation~$\dd_i$ falling on the Jacobiator~$\Jac_\cP(a,b,c)$, the Leibniz rule yields that
\[
\dd_i\bigl(\Jac_\cP(a,b,c)\bigr)=\bigl(\dd_i\bigl(\Jac(\cP)\bigr)\bigr)(a,b,c) +\Jac_\cP(\dd_i a,b,c)
+\Jac_\cP(a,\dd_i b, c) +\Jac_\cP(a,b,\dd_i c).
\]
The last three terms in the right\/-\/hand side of the above formula amount to a redefinition of Jacobiator's arguments; hence every such term vanishes. Consequently, the first term in which the derivation~$\dd_i$ acts on the two internal vertices of the Jacobiator itself is equal to zero: $\bigl(\dd_i\bigl(\Jac(\cP)\bigr)\bigr)(\cdot,\cdot,\cdot)=0$. One now proceeds recursively over arbitrarily large finite set of derivations that sumltaneously fall on the Jacobiator, then acting independently from each other according to the Leibniz rule.

\begin{define}%\marginpar{Draw}
A \emph{Leibniz graph} is a graph whose vertices are either sinks, or the sources for two arrows, %picture
or the Jacobiator (which is a source for three arrows); %picture
there must be at least one Jacobiator vertex.
The three arrows originating from a Jacobiator vertex must land on three distinct vertices (and not on the Jacobiator itself).\footnote{Each edge falling on a Jacobiator works by the Leibniz rule on the two internal vertices in~it. Combined with expansion~\eqref{EqJacFig} of the Jacobiator using graphs, this tells us that every Leibniz graph expands to a sum of Kontsevich graphs which were introduced before.}
\end{define}

\begin{example}
An example of a Leibniz graph is given in Fig.~\ref{FigSample}.%
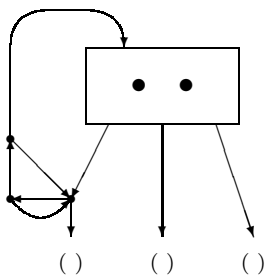
\begin{figure}[htb]
\begin{minipage}{0.3\textwidth}
\begin{align*}
\unitlength=1mm
\special{em:linewidth 0.4pt}
\linethickness{0.4pt}
\begin{picture}(40.67,35.00)
\put(15.00,20.00){\framebox(20.00,10.00)[cc]{$\bullet\quad\bullet$}}
\put(25.00,20.00){\vector(0,-1){15.00}}
\put(18.00,20.00){\vector(-1,-2){5.00}}
\put(32.00,20.00){\vector(1,-3){5.00}}
\put(13.00,10.00){\vector(0,-1){5.00}}
\put(13.00,10.00){\vector(-1,0){8.00}}
\put(13.00,0.00){\makebox(0,0)[cb]{\tiny(\ )}}
\put(25.00,0.00){\makebox(0,0)[cb]{\tiny(\ )}}
\put(5.00,10.00){\vector(0,1){8.00}}
\put(5.00,18.00){\vector(1,-1){8.00}}
\put(37.00,0.00){\makebox(0,0)[cb]{\tiny(\ )}}
\bezier{100}(5.00,10.00)(9.00,5.00)(13.00,10.00)
\put(13.00,10.00){\vector(1,1){0.00}}
\put(5.00,18.00){\line(0,1){12}}
\bezier{80}(5.00,30.00)(5.00,35.00)(10.00,35.00)
\put(10.00,35.00){\line(1,0){5.00}}
\bezier{80}(15.00,35.00)(20.00,35.00)(20.00,30.00)
\put(20.00,30.00){\vector(0,-1){0.00}}
\put(13.00,10.00){\circle*{1}}
\put(5.00,10.00){\circle*{1}}
\put(5.00,18.00){\circle*{1}}
\end{picture}
\end{align*}
\end{minipage}
%\hspace{-15mm}
\small
\begin{minipage}{0.4\textwidth}
\begin{itemize}
\item There is a cycle, 
\item there is a loop, 
\item there are no tadpoles in this graph,
\item an arrow falls back on $\Jac(\cP)$, 
\item and $\Jac(\cP)$ does not stand on all of the three sinks.
\end{itemize}
\end{minipage}
\normalsize
\caption{An example of Leibniz graph.}
\label{FigSample}
\end{figure}
\end{example}

%Every Leibniz graph can be expanded to a sum of Kontsevich graphs, by expanding both the Leibniz rule(s) and all copies of the Jacobiator. %; e.g. see \eqref{EqStayVanish}.
%In this way (sums of) 
Leibniz graphs %also 
encode (poly)dif\-fe\-re\-ntial operators $\Diamond(\cP, \Jac(\cP))$ whose arguments are at least one copy of the tri\/-\/vector~$\Jac(\cP)$ and possibly, the bi\/-\/vector~$\cP$ itself.\footnote{In Example~\ref{ExFactorH2} the Leibniz graph amounts to just one tri\/-\/vector vertex and no extra copies of the Poisson bi\/-\/vector in other internal vertices, of which there are none.}

\begin{proposition}\label{PropLeibnizGraphZero}
For every Poisson bi\/-\/vector~$\cP$ the value --\,at the Jacobiator $\Jac(\cP)$\,-- of every (poly)\/dif\-fe\-ren\-ti\-al operator encoded by the Leibniz graph(s) is zero.
\end{proposition}

\begin{cor}
To prove that the associator for the Kontsevich star\/-\/product~$\star_\hbar$ vanishes for every Poisson structure contained in each internal vertex within a graph expansion of~$\star_\hbar$,
%a sum of Kontsevich graphs vanishes at every Poisson structure, 
it suffices to realize the associator as a sum of Leibniz graphs:\footnote{The same technique, showing the vanishing of a sum of Kontsevich graphs by writing it as a sum of Leibniz graphs, has been used in~\cite{f16} to solve another problem in the graph calculus.}
\begin{equation}\tag{\ref{EqDefDiamond}}
\Assoc_{\star_\hbar}(f,g,h) =
 %\bigl(f\mathbin{{\star}_\hbar}g\bigr)\mathbin{{\star}_\hbar}h -
 %f\mathbin{{\star}_\hbar}\bigl(g\mathbin{{\star}_\hbar}h\bigr) =
\Diamond\,\bigl(\cP,\Jac_\cP(\cdot,\cdot,\cdot)
 %\{\{f,g\}_{\cP},h\}_{\cP}+\text{c.\,p.}
\bigr)(f,g,h). %, \qquad f,g,h\in C^\infty(N^n)[[\hbar]],
\end{equation}
From Example~\ref{ExFactorH2} we already know that the factorizing polydifferential operator in~\eqref{EqDefDiamond} is $\Diamond=\tfrac{2}{3}\,\mathbf{1}+\bar{o}(1)$.
\end{cor}
%In [3] this computer-assisted scheme of reasoning and the corresponding software were applied to
%solution of a similar factorization problem in the Kontsevich graph calculus.

\begin{example}
The assemply of factorizing operator~$\Diamond\mod\bar{o}(\hbar)$, i.e.\ at order~$3$ in the expansion $\Assoc_{\star_\hbar}(\cdot,\cdot,\cdot)\mod\bar{o}(\hbar^3)$, is explained in~\cite{sqs15}; 
linear in its argument at~$\hbar^1$, the operator~$\Diamond\mod \bar{o}(\hbar)$ has differential order one with respect to the Jacobiator.

The next step $\Diamond\mod\bar{o}(\hbar^2)$ in factorization~\eqref{EqDefDiamond}, now at order~$4$ with respect to~$\hbar$ in the associator, is achieved in~\cite{cpp}.
\end{example}

\begin{state}[\cite{cpp}]\label{PropTwoEdges}
No solution~$\Diamond\mod\bar{o}(\hbar^2)$ of factorization problem~\eqref{EqDefDiamond} can have differential order less than two with respect ot the Jacobiator~$\Jac(\cP)$; conversely, there always exists a Leibniz graph at~$\hbar^2$ in the polydifferential operator~$\Diamond$ such that at least two arrows fall on the Jacobiator.
\end{state}
%\\[2pt]
%\centerline{\rule{1in}{0.7pt}}

\section{Deformation quantisation ${\times}\mapsto{\star}_\hbar$ in the algebras~$\bcA$ of local functionals for field models}\label{SecInfinite}
\noindent%
In this section we lift the Kontsevich graph technique from a quantisation ${\times}\mapsto{\star}_\hbar$ of the product~$\times$ in the algebra~$\cA$ of smooth functions on a finite\/-\/dimensional affine Poisson manifold~$\bigl(N^n$,\ ${\cP}\bigr)$ to the deformation $\times\mapsto\star_\hbar$ of the product of local functionals in the geometry of $N^n$-\/valued %physical %(gauge) 
fields over an affine base manifold~$M^m$. 
We shall analyse the construction of local variational polydifferential operators which are encoded by the Kontsevich graphs (in particular, by the Leibniz graphs), now containing at each internal vertex a copy of the variational Poisson structure~$\{\cdot,\cdot\}_{\bcP}$. %over the jet space~$J^\infty(\pi)$.
It is the Gel'fand formalism of singular linear integral operators supported on the diagonal~\cite{GelfandShilov} that becomes our working language.

\subsection{Field model geometry}\label{S3.1}%3.1.
%\addtocontents{toc}{\hbox to\textwidth{\strut\qquad 3.1. Field model geometry
%\hfill \pageref{S3.1}}}
To extend the affine geometry of section~\ref{SecFinite}, let us list the ingredients of the affine bundle set\/-\/up.\footnote{In retrospect, the construction in section~\ref{SecFinite} can be viewed as a special case of such ``bundles'' over a point~$M^0$.}

Let $\bigl(M^m$,\ $\dvol(\cdot)\bigr)$ be an $m$-\/dimensional oriented affine real manifold equipped with a volume element.%
\footnote{Not excluding the case where the volume element~$\dvol(\bx)$ can nontrivially depend on the jets~$j^\infty_{\bx}(\phi)$ of sections~$\phi\in\Gamma(\pi)$ over points~$\bx\in M^m$, for the sake of brevity let us
% nevertheless --\,\,-- %abstain from writing
not write such %possible
admissible second argument in~$\dvol\bigl(\cdot,j_\infty(\phi)(\cdot)\bigr)$.}
%%%
Let $\pi\colon E^{m+n}\to M^m$ be an affine bundle with $n$-\/dimensional fibres~$N^n$ over the base~$M^m$.
Denote by $\bu=(u^1$,\ $\ldots$,~$u^n)$ an $n$-\/tuple of local coordinates in the fibre~$N^n$.
%for~$1\leqslant n<\infty$.

Denote by $J^\infty(\pi)$ the total space of the bundle~$\pi_\infty$ of infinite jets~$j^\infty(\bs)(\cdot)$ for sections~$\bs\in\Gamma(\pi)$ of the bundle~$\pi$ over~$M^m$; the infinite jet space~$J^\infty(\pi)$ is the projective limit $\projlim_{k\to+\infty} J^k(\pi)$ of the sequence
of finite jet spaces~$J^k(\pi)$,
\[
M^m\xleftarrow{\:\pi\:} E^{m+n}=J^0(\pi)\gets J^1(\pi)\gets\ldots\gets J^k(\pi)\gets\ldots\gets J^\infty(\pi).
\]
It is clear that affine reparametrisations $\widetilde{\bx}(\bx)$ of local coordinates on the base~$M^m$ induce linear transformations of smooth sections' derivatives up to %a %every
positive order~$k$ for all~$k>0$. By definition, we put~$[\bu]$ for an object's dependence on sections~$\bs$ and their derivatives up to arbitrarily large but still finite order, which is well defined by the above.

Denote by $\bar{H}^m(\pi)$ the vector space of integral functionals $\Gamma(\pi)\to\Bbbk$ of form $F=\int f\bigl(\bx_1,[\bu]\bigr)\cdot\dvol(\bx_1)$ such that $F(\bs)=\int_{M^m} f\bigl(\bx_1,j^\infty_{\bx_1}(\bs)\bigr)\cdot\dvol(\bx_1)$.
Viewed as functionals~$\Gamma(\pi)\to\Bbbk$ that take sections~$\phi\in\Gamma(\pi)$ %of $N^n$-valued physical fields
over~$M^m$ to numbers, the integral objects $F,G,H\in\bar{H}^m(\pi)
   %\hookrightarrow\ov{\mathfrak{M}}^m(\pi)\ni F\times G
$ can be shifted by using the null functionals $Z\colon\Gamma(\pi)\to0\in\Bbbk$. 
Those can be of topological nature,\footnote{\label{FootNullLag}%
For instance, set~$m=1$, let~$M^m\mathrel{{:}{=}}\BBS^1\cup\BBS^1$, take the usual angle variables
$\varphi_1,\varphi_2\colon\BBR^1\to\BBS^1$ on the two circles, and consider the null Lagrangian
$\cL=\int\Id\varphi_1-\int\Id\varphi_2$ that takes \textsl{every} section of an affine bundle~$\pi$ over such~$M^1$ to $2\pi-2\pi=0\in\Bbbk$. Obviously, the cohomology class~$\cL$ in~$H^1(\pi)$ is nonzero for the top\/-\/degree form
$\Id\varphi_1-\Id\varphi_2$; for it is only locally but not globally exact.}
%%%
$Z\in H^m(\pi)$. We always quotient them out in this paper by taking the factorgroup~$\bar{H}^m(\pi)/H^m(\pi)$. 
Secondly,
%the 
null integral functionals~$\Gamma(\pi)\to0\in\Bbbk$ can mark the zero class
$\int\Id_h(\Theta)\cong\int0\in\bar{H}^m(\pi)$ in the top\/-\/degree %senior 
horizontal cohomology group%
\footnote{\label{FootNoBoundary}%
The integrations by parts~$\cong$ over~$M^m$ are nominally present in the construction of horizontal cohomolohgy groups~$\bar{H}^i(\pi)$ for the jet space~$J^{\infty}(\pi)$ over the bundle~$\pi$. %of physical fields; 
%referring to~\S%~3.3.4=
%\ref{SecWhyLeaks} (cf.~\cite{gvbv,Larnaca14}), we advise supreme caution in doing that~--- if doing at~all.
%%%
By default, let us technically assume that no boundary terms would ever appear from~$\dd M^m$ in any formulae. 
For that, either let the base~$M^m$ be a closed manifold (hence~$\dd M^m=\varnothing$; for instance, take~$M^1=\BBS^1$) or choose the class~$\Gamma(\pi)$ of \textsl{admissible} sections for the bundle~$\pi$ is such a way that they decay so rapidly towards the boundary~$\dd M^m$
%vanish at $\dd M^m$ together with a
%sufficient number of their one-side derivatives -- enough for 
that all the integrands under study also vanish at~$\dd M^m$.
For instance, suppose that~$M^1=\BBR$ and all the field profiles~$\bu=\phi(\bx)$ are Schwarz.%
%rapidly decay towards the spatial infinity.
%Still let us warn the reader against idle, unmotivated %/ unnecessary / redundant
%integrations by parts.
%%%
%Strange though it may seem, almost all of the above is irrelevant in practice (so that indeed, much physics can be expressed %/ realised]
%by the boundary values at $\dd M^m$, cf.~\cite{SeibergWitten}). %\marginpar{Ref} 
%An overwhelming majority of integrations by parts which we
%addressed in \S%~3.2~%IteratedVariations
%\ref{SecElements} are performed over the supports
%$\supp\delta\bs\bigl(\,\cdot\,,\bs(\,\cdot\,)\bigr)\subseteq M^m$ of fields' $\bs\in\Gamma(\pi)$ virtual excitations. Personifying
%the concept of physical fields, these variations must vanish at their supports' boundaries together with all their 
%derivatives, i.\,e., $\left.j_{\infty}(\delta\bs)\right|_{\dd(\supp\delta\bs)}=0$, while such domains themselves can be
%taken arbitrarily small.
}
% Integration by parts over $M^m$ would ``bring" the [remote]
% boundary values to a given point $\bx\in M^m$ (in no time).
for $J^{\infty}(\pi)$ over $M^m$. 

%\begin{rem}
%From what follows %in~\S\ref{...}\marginpar{}
%it will be readily seen that 
%Functionals~$F_1$ and~$F_2$ such that $F_1-F_2\colon\Gamma(\pi)\to0\in\Bbbk$ can still contribute differently to the tails of quantisation series %serii
%beyond the leading deformation term.\footnote{This effect of nontrivial synonyms of zero also shows up in other implementations of the geometry of iterated variations, e.g., in the Batalin\/--\/Vilkovisky formalism (see~\cite{gvbv}).}
%(cf.\ footnotes~\ref{FootNullLag} and~\ref{FootNoBoundary} on p.~\pageref{FootNullLag}). 
%\end{rem}

By brute force, introduce the multiplication ${\times}\colon F\otimes G\mapsto
F\times G=\int f\bigl(\bx_1,[\bu]\bigr)\cdot\dvol(\bx_1)\times\int g\bigl(\bx_2,[\bu]\bigr)\cdot\dvol(\bx_2)=\iint f\bigl(\bx_1,[\bu]\bigr)\times g\bigl(\bx_2,[\bu]\bigr)\cdot\dvol(\bx_1)\,\dvol(\bx_2)\colon\Gamma(\pi)\to\Bbbk$ for
$G=\int g\bigl(\bx_2,[\bu]\bigr)\cdot\dvol(\bx_2)$. This yields the algebra~$\bcA$ of \textsl{local functionals},%
\footnote{We recall that in the (graded-)\/commutative set\/-\/up one has that
$\smash{\bigl(F\mathbin{\stackrel{\overline{\mathfrak{M}}^m(\pi)}{\times}}G\bigr)(\bs)=F(\bs)\mathbin{\stackrel{\Bbbk}{\times}}G(\bs)}$ but a known %geometric %algebraic  %%%prevents 
mechanism destroys this algebra homomorphism in a larger setting of formal noncommutative variational symplectic geometry and its calculus of cyclic words~(\cite{cycle16}, cf.~\cite{KontsevichCyclic,OlverSokolov1998}).}
%%%
also denoted by~$\overline{\mathfrak{M}}^m(\pi)$ in~\cite{gvbv,cycle16}

Referring only to the fibre's local portrait but not to its global organisation, we introduce the $\BBZ_2$-\/parity odd coordinates $\bxi=(\xi_1$,\ $\ldots$,~$\xi_n)$ in the reversed\/-\/parity cotangent spaces $\Pi T^*_{(\bx,\bs(\bx))}N^n$ to the fibres $N^n\simeq \pi^{-1}(\bx)$ of the bundle~$\pi$, see~\cite[\S2.1]{cycle16} and~\cite[\S2.1]{gvbv}. Let us note that for \textsl{vector} spaces~$N^n=\BBR^n$, the vector space isomorphism $T_{(\bx,\bs(\bx))}N^n\simeq N^n$ reduces this construction of Kupershmidt's variational cotangent bundle~\cite{KuperCotangent} over~$\pi_\infty\colon J^\infty(\pi)\to M^m$ to the Whitney sum~$J^\infty\bigl(\pi\mathbin{{\times}_M}\Pi\widehat{\pi}\bigr)$. 

\begin{convention}
The notation $\pi\mathbin{{\times}_M}\Pi\widehat{\pi}$ will be used in what follows to avoid an agglomeration of formulae. Indeed, the case of affine bundle~$\pi$ already impels the construction of horizontal jet bundle~$\overline{J^\infty_{\pi_\infty}}(\Pi T^*\pi)$ over the space~$J^\infty(\pi)$. %, see~\cite{Topical}.
\end{convention}

The variational bi\/-\/vectors $\bcP\in\bar{H}^m\bigl(\pi\mathbin{{\times}_M}\Pi\widehat{\pi}\bigr)$ are integral functionals of the form
\[
\bcP=\tfrac{1}{2}\int\langle\bxi\cdot A{\bigr|}_{(\bx,[\bu])}(\bxi)\rangle =
\tfrac{1}{2}\int\xi_i\,P^{ij}_\tau(\bx,[\bu])\,\xi_{j,\tau}\cdot\dvol(\bx),
\]
where the linear total differential operators $A=\bigl\| P^{ij}_\tau\cdot\bigl(\tfrac{\Id}{\Id\bx}\bigr)^\tau \bigr\|_{i=1,\ldots,n}^{j=1,\ldots,n}$ are skew\/-\/adjoint (to make the object~$\bcP$ well defined); for all multi\/-\/indexes~$\tau$, the parity\/-\/odd symbols $\xi_{j,\varnothing}=\xi_j$,\ $\xi_{j,x^k}$, $\xi_{j,x^k x^\ell}$,\ $\ldots$,\ $\xi_{j,\tau}$,\ $\ldots$ are the respective jet fibre coordinates.

The construction of variational $k$-\/vectors with $k\geqslant0$ is alike, see~\cite{cycle16}. Due to the introduction of parity\/-\/odd variables~$\bxi$ as canonical conjugates of the $n$-\/tuples~$\bu$, the vector space of all variational multivectors is naturally endowed with the parity\/-\/odd variational Poisson bracket, or variational \textsl{Schouten bracket}~$\lshad\cdot,\cdot\rshad$. Its construction --\,as descendent %adj.; noun="descendant".
structure with respect to the Batalin\/--\/Vilkovisky Laplacian~$\Delta$\,-- was recalled in~\cite{gvbv,cycle16}.
%for consistency, we shall discuss the composition of~$\lshad\cdot,\cdot\rshad$ in what follows (see p.~\pageref{pVariationsArrows} below).

\begin{define}
A variational bi\/-\/vector~$\bcP$ is called \textsl{Poisson} if it satisfies the classical master\/-\/equation~$\lshad\bcP,\bcP\rshad\cong0$.
The horizontal cohomology class equivalence ${}\cong0$ means, in particular, that the variational tri\/-\/vector~$\lshad\bcP,\bcP\rshad$, viewed as an integral functional, takes~$\Gamma(\pi)\to 0\in\Bbbk$. 
\end{define}

Every variational Poisson bi\/-\/vector~$\bcP$ induces the respective variational Poisson bracket~$\{\cdot,\cdot\}_{\bcP}\colon\bar{H}^m(\pi)\times\bar{H}^m(\pi)\to\bar{H}^m(\pi)$ on the space of integral functionals~$\Gamma(\pi)\to\Bbbk$. An axiomatic construction of~$\{\cdot,\cdot\}_{\bcP}$ is explained in Definition~\ref{DefVarPBr} on p.~\pageref{DefVarPBr}; it is the \textsl{derived bracket} $\lshad\lshad\bcP,\cdot\rshad,\cdot\rshad$ of two Hamiltonians (see~\cite[\S3]{cycle16}).%
\footnote{Note that an attempt to modify the volume element $\dvol(\cdot)$ 
on~$M^m$ can affect the output of~$\{\cdot,\cdot\}_{\bcP}$.}

The bracket~$\{\cdot,\cdot\}_{\bcP}$ is extended, via the Leibniz rule, from the vector space~$\bar{H}^m(\pi)$ of integral functionals~$H_1$,\ $H_2$,\ $\ldots$ to the Poisson structure on the algebra~$\bcA$
%$\overline{\mathfrak{M}}^m(\pi)$ 
of (sums of) such functionals' formal products~$H_1\times\ldots\times H_\ell\colon\Gamma(\pi)\to\Bbbk$.
%This is done by using the graded Leibniz rule that extends the variational Schouten bracket~$\lshad\cdot,\cdot\rshad$ from the vector space~$\bar{H}^m\bigl(\pi\mathbin{{\times}_M}\Pi\widehat{\pi}\bigr)$ of variational multivectors to the differential graded Lie algebra~$\overline{\mathfrak{M}}^m\bigl(\pi\mathbin{{\times}_M}\Pi\widehat{\pi}\bigr)$, which obviously contains the vector space~$\overline{\mathfrak{M}}^m(\pi)$ as its zero\/-\/grading component. 

\begin{rem}
The value of~$\{\cdot,\cdot\}_{\bcP}$ in~$\bar{H}^m(\pi)$ at two integral functionals does not depend on a choice of representatives for the two arguments and for the variational Poisson bi\/-\/vector~$\bcP\in\bar{H}^m\bigl(
\pi\mathbin{{\times}_M}\Pi\widehat{\pi}\bigr)$, taken modulo those 
integral functionals which map all sections of the respective (super)\/bundle 
to~$0\in\Bbbk$. This is no longer necessarily so for the higher\/-\/order terms, beyond the variational Poisson bracket~$\{\cdot,\cdot\}_{\bcP}$ at~$\hbar^1$, in expansions~\eqref{EqUniversalStar}.
\end{rem}

\begin{rem}
Let us remember that every integral functional --\,e.g., taken as a building block in a local functional\,-- does carry its own integration variable which runs through that integral functional's own copy of the base~$M^m$ for the respective (super)\/bundle. For %instance
a given model over~$\bigl(M^m$,\ $\dvol(\cdot)\bigr)$,
the variational Poisson bi\/-\/vector $\bcP=\tfrac{1}{2}\int\xi_i\,P^{ij}_\tau(\bx,[\bu])\,\bigl(\tfrac{\Id}{\Id\bx}\bigr)^\tau(\xi_j)\cdot\dvol(\bx)$ and two Hamiltonians, $F=\int f(\bx_1,[\bu])\cdot\dvol(\bx_1)$ and $G=\int g(\bx_2,[\bu])\cdot\dvol(\bx_2)$, are integral functionals defined at sections of the bundles $\pi\mathbin{{\times}_M}\Pi\widehat{\pi}$ and~$\pi$, respectively.
In total, these three objects carry \textsl{three} copies of the given volume element~$\dvol(\cdot)$ on~$M^m$.

On the other hand, the variational Poisson bracket~$\{F,G\}_{\bcP}$ of~$F$ and~$G$ with respect to~$\bcP$ is an integral functional $\Gamma(\pi)\to\Bbbk$ that carries \textsl{one} copy of the volume element. Why and where to have the two copies of~$\dvol(\cdot)$ gone\,? We now recall an answer to this question.
\end{rem}

\subsection{Elements of the geometry of iterated variations}\label{SecElements}%\label{S3.2}%\S3.2. 
%\addtocontents{toc}%
%{\hbox to\textwidth{\strut\qquad \protect\ref{SecElements}. %3.2
%Elements of the geometry of iterated variations\hfill \protect\pageref{SecElements}}}

\subsubsection{}%{\S3.2.1.}
Let $(\bs,\bs^{\dagger})$ be a two\/-\/component section of the Whitney sum $\pi\times_M\Pi\hat{\pi}$ of bundles. Suppose that this section undergoes an infinitesimal shift 
along the direction
\[
(\delta\bs,\delta\bs^{\dagger})\bigl(\bx,\bs(\bx),\bs^{\dagger}(\bx)\bigr)=
\sum\limits_{i=1}^n\left(\delta\bs^i(\bx)\cdot\vec{e}_i(\bx)+
\delta\bs_i^{\dagger}(\bx)\cdot\vec{e}^{\mathstrut\ \dagger,i}(\bx)\right),
\]
which we decompose with respect to the adapted basis $(\vec{e}_i,\vec{e}^{\mathstrut\,\dagger,i})$
in the tangent space $T_{(\bx,\bs(\bx))}\pi^{-1}(\bx)\oplus T_{%(\bx,
\bs^{\dagger}(\bx)%)
}T^*_{(\bx,\bs(\bx))}\pi^{-1}(\bx)$.
At their attachment point, the vectors $\vec{e}_i$ and $\vec{e}^{\mathstrut\ \dagger,j}$ are --~by definition~--
tangent to the 
respective coordinate lines for variables~$u^i$ and~$\xi_j$. By construction, these vectors 
$\vec{e}_i$ and $\vec{e}^{\mathstrut\ \dagger,j}$ are dual; at every $i$ running from~$1$ to~$n$,
the two \textsl{ordered} couplings of
(co)vectors attached over $\bx\in M^m$ at the fibres' points --\,with values
$\bs(\bx)$ and $\bs^{\dagger}(\bx)$ of the respective coordinates\,-- are
\begin{equation}\label{EqTwoCouplings}
\langle
\stackrel{\text{first}}{\vec{e}_i},
\stackrel{\text{second}}{\vec{e}^{\mathstrut\,\dagger,i}}\rangle=+1
\quad\text{ and }\quad
\langle\stackrel{\text{first}}{\vec{e}^{\mathstrut\,\dagger,i}},\stackrel{\text{second}}{\vec{e}_i}\rangle=-1.
\end{equation}
Likewise, the coefficients
$\delta s^i(\,\cdot\,,\bs(\,\cdot\,))$ and $\delta s_i^{\dagger}(\,\cdot\,,\bs(\,\cdot\,),\bs^{\dagger}(\,\cdot\,))$
of the virtual shifts along the $i^{\text{th}}$ coordinate lines $u^i$ and $\xi_i$ are normalised %by 
via
\begin{equation}\label{EqNormalise}
\delta s^i(\bx,\bs(\bx))\cdot\delta s^{\dagger}_i(\bx,\bs(\bx),\bs^{\dagger}(\bx))\equiv1 \qquad\text{ (no summation!)}
\end{equation}
over all internal points $\bx\in\supp\delta s^i\subseteq M^m$. % (see footnote~\ref{FootNonPhysicalOdd}).
%It is precisely this mathematical construction in terms of which the physical idea of fields as degrees of freedom attached at every point of the space-time is expressed.
%Note that the support of virtual shift $\delta\bolds$ can be arbitrarily small.

\begin{convention}%By %construction / convention, 
The differentials of functionals' densities are %[always]
expanded with respect to the bases~$\vec{e}^{\mathstrut\,\dagger,j}$, $\vec{e}_i$ in the fibres tangent spaces;
the plus or minus signs in the sections' %offsets / 
shifts are chosen in such a way that the couplings always evaluate to~$+1$.
\end{convention}

The directed variations $\overrightarrow{\delta\bs}$ and $\overrightarrow{\delta\bs^{\dagger}}$, as well as 
$\overleftarrow{\delta\bs}$ and $\overleftarrow{\delta\bs^{\dagger}}$, are singular linear integral operators supported, due 
to~\eqref{EqTwoCouplings}, on the diagonal. 
Each variation contains %[therefore]
$n$~copies of Dirac's $\boldsymbol{\delta}$-\/distribution weighted by the respective coefficients $\delta s^i$ and $\delta s_i^{\dagger}$. We have that
\begin{align*}
%\mathstrut&
\overrightarrow{\delta\bs}=&
\int\Id\bby\,\Bigl{\langle}(\delta s^i)\Bigl(\frac{\overleftarrow{\dd}}{\dd\bby}\Bigr)^{\sigma}(\bby)\cdot
\underrightarrow{
\stackrel{\text{first}}{\vec{e}_i(\bby)}|\stackrel{\text{second}}{\vec{e}^{\mathstrut\,\dagger,i}(\,\cdot\,)}
}
\Bigr{\rangle}\,\frac{\overrightarrow{\dd}}{\dd u^i_{\sigma}},
%\\[-15pt]
%\mathstrut&
%\hphantom{*\ %
%\overrightarrow{\delta\bs}=
%\int\bby\Bigl{\langle}(\delta s^i)\Bigl(\frac{\overleftarrow{\dd}}{\dd\bby}\Bigr)^{\sigma}(\bby)\cdot
%}
%\overrightarrow{
%\hphantom{
%\stackrel{\text{first}}{\vec{e}_i(\bby)}|\stackrel{\text{second}}{\vec{e}^{\mathstrut\ \dagger,i}(\,\cdot\,)}
%}
%}
\\
%\mathstrut&
\overrightarrow{\delta\bs^{\dagger}}=&
\int\Id\bz\,\Bigl{\langle}(\delta s^{\dagger}_i)\Bigl(\frac{\overleftarrow{\dd}}{\dd\bz}\Bigr)^{\sigma}(\bz)\cdot
\underrightarrow{
(\stackrel{\text{first}}{-\vec{e}^{\mathstrut\,\dagger,i})(\bz)}
|\stackrel{\text{second}}{\vec{e}_i(\,\cdot\,)}
}
\Bigr{\rangle}\,\frac{\overrightarrow{\dd}}{\dd\xi_{i,\sigma}},
%\\[-15pt]
%\mathstrut&
%\hphantom{*
%\overrightarrow{\delta\bs^{\dagger}}=
%\int\Id\bby\Bigl{\langle}(\delta s^{\dagger}_i)\Bigl(\frac{\overleftarrow{\dd}}{\dd\bby}\Bigr)^{\sigma}(\bby)\cdot
%}
%\overrightarrow{
%\hphantom{
%(-\stackrel{\text{first}}{\vec{e}^{\mathstrut\ \dagger,i})(\bby)}
%|\stackrel{\text{second}}{\vec{e}_i(\,\cdot\,)}
%}
%}
\\
%\mathstrut&
\overleftarrow{\delta\bs}=&
\int\Id\bby\,\frac{\overleftarrow{\dd}}{\dd u^i_{\sigma}}\,
\Bigl{\langle}
\underleftarrow{
\stackrel{\text{second}}{\vec{e}^{\mathstrut\,\dagger,i}(\,\cdot\,)}|\stackrel{\text{first}}{\vec{e}_i(\bby)}
}
\cdot
\Bigl(\frac{\overrightarrow{\dd}}{\dd\bby}\Bigr)^{\sigma}(\delta s^i)(\bby)
\Bigr{\rangle},
%\\[-15pt]
%\mathstrut&
%\hphantom{*\ %
%\overleftarrow{\delta\bs}=
%\int\Id\bby\frac{\overleftarrow{\dd}}{\dd u^i_{\sigma}}
%}
%\overleftarrow{
%\hphantom{
%\stackrel{\text{second}}{\vec{e}^{\mathstrut\,\dagger,i}(\,\cdot\,)}|\stackrel{\text{first}}{\vec{e}_i(\bby)}
%}
%}
%\hphantom{
%\cdot
%\Bigl(\frac{\overleftarrow{\dd}}{\dd\bby}\Bigr)^{\sigma}(\delta s^i)(\bby)
%\Bigr{\rangle},
%}
\\
%\mathstrut&
\overleftarrow{\delta\bs^{\dagger}}=&
\int\Id\bz\,\frac{\overleftarrow{\dd}}{\dd\xi_{i,\sigma}}\,
\Bigl{\langle}
\underleftarrow{
\stackrel{\text{second}}{\vec{e}_i(\,\cdot\,)}|\stackrel{\text{first}}{(-\vec{e}^{\mathstrut\ \dagger,i})(\bz)}
}
\cdot
\Bigl(\frac{\overrightarrow{\dd}}{\dd\bz}\Bigr)^{\sigma}(\delta s^{\dagger}_i)(\bby)
\Bigr{\rangle},
%\\[-15pt]
%\mathstrut&
%\hphantom{*\!%
%\overleftarrow{\delta\bs^{\dagger}}=
%\int\Id\bby\frac{\overleftarrow{\dd}}{\dd\xi_{i,\sigma}}
%\Bigl{\langle}
%}
%\overleftarrow{
%\hphantom{
%\stackrel{\text{second}}{\vec{e}_i(\,\cdot\,)}|\stackrel{\text{first}}{(-\vec{e}^{\mathstrut\ \dagger,i})(\bby)}
%}
%}
%\hphantom{
%\cdot
%\Bigl(\frac{\overleftarrow{\dd}}{\dd\bby}\Bigr)^{\sigma}(\delta s^{\dagger}_i)(\bby)
%\Bigr{\rangle},
%}
\end{align*}
see~\cite[\S2.2--3]{gvbv} for details; for brevity, the indication of fibre points for given $\bs(\cdot)$ and
$\bs^{\dagger}(\cdot)$ is omitted in such formulae. 
Whenever acting %on their space 
on %the spaces of 
local functionals, %which were discussed in the preceding section, 
these linear operators yield those functionals' responses to infinitesimal shifts of their arguments. 
i.\,e.\ of the sections at which the functionals are evaluated.

Let us now recall the mechanism of integration by parts (see~\cite{gvbv} and~\cite[\S2.5]{cycle16}).

\begin{lemma}
In absence of boundary terms, the on\/-\/the\/-\/diagonal integration by parts converts derivatives along one copy of the integration domain~$M^m$ into $(-1)\times$~derivatives with respect to the same variables, now referred to another copy of the base.% manifold.
\end{lemma}

\begin{proof}
Consider a point~$\bby$ of the affine manifold~$M^n$ and denote by $\bby+\delta\bby\in M^n$ a near\/-\/by point with coordinates~$y^i+\delta y^i$, here and immediately below $1\leqslant i,\alpha\leqslant n$; the notation $\lim_{\delta\bby\to 0}$ makes obvious sense. For the sake of brevity, put $\sigma\mathrel{{:}{=}}\{x^\alpha\}$. We have that, due to the absence of boundary terms and then by definition,\footnote{\label{FootTD}%
The definition of total derivative $\Id/\Id x$, which is
\[
\Bigl(j^{\infty}(s)^*\Bigl(\Bigl(\frac{\Id}{\Id x}f\Bigr)(x,[u])\Bigr)\Bigr)(x_0)\stackrel{\text{def}}{=}
\Bigl(\frac{\dd}{\dd x}\Bigl(j^{\infty}(s)^*\Bigl(f(x,[u])\Bigr)\Bigr)\Bigr)(x_0),
\]
explains why the partial derivatives $\dd/\dd x$ reshape into $\Id/\Id x$ as soon as they arrive to the graph's vertices and
there, they act on the objects $f$ which are defined over jet bundles and which %, in retrospect, 
are evaluated at the infinite jets~$j^{\infty}(s)$ of sections~$s$.}
%%%
\begin{align*}%multline*}
\int&\Id\bby \bigl\langle(\delta s^i)\frac{\overleftarrow{\dd}}{\dd y^\alpha}(\bby)\cdot
\underrightarrow{\vec{e}_i(\bby),\vec{e}^{\,\dagger,i}(\bx)}\,
\frac{\overrightarrow{\dd}}{\dd u^i_{x^\alpha}} f(\bx,[\bu],[\bxi])\bigr|_{j^\infty_\bx(\bolds,\bolds^\dagger)}\bigr\rangle={}
\\
{}&=\int\Id\bby\ \delta s^i(\bby)\bigl(-\frac{\overrightarrow{\dd}}{\dd y^\alpha}\bigr)
\bigl\langle\underrightarrow{\vec{e}_i(\bby),\vec{e}^{\,\dagger,i}(\bx)}\,
\frac{\overrightarrow{\dd}}{\dd u^i_{x^\alpha}} f(\bx,[\bu],[\bxi])\bigr|_{j^\infty_\bx(\bolds,\bolds^\dagger)}\bigr\rangle
\\
{}&\stackrel{\text{def}}{=} -\int\Id\bby\ \delta s^i(\bby)\,\lim\limits_{\delta y^\alpha\to+0}\frac{1}{\delta y^\alpha} 
\left\{
\begin{aligned}
\bigl\langle & \underbrace{
\vec{e}_i(\bby+\delta y^\alpha),\vec{e}^{\,\dagger,i}(\bx)
}_{\text{$+1$ if $\bx=\bby+\delta y^\alpha $}}\,
\frac{\overrightarrow{\dd}}{\dd u^i_{x^\alpha}} f(\bx,[\bu],[\bxi])\bigr|_{j^\infty_\bx(\bolds,\bolds^\dagger)}\bigr\rangle
\\
&-\bigl\langle\underbrace{\vec{e}_i(\bby),\vec{e}^{\,\dagger,i}(\bx)}_{\text{$+1$ if $\bx=\bby$}}\,
\frac{\overrightarrow{\dd}}{\dd u^i_{x^\alpha}} f(\bx,[\bu],[\bxi])\bigr|_{j^\infty_\bx(\bolds,\bolds^\dagger)}\bigr\rangle
\end{aligned}
\right\}
\\
{}&\stackrel{\text{def}}{=} \int\Id\bby\ \delta s^i(\bby)
\bigl\langle\underrightarrow{\vec{e}_i(\bby),\vec{e}^{\,\dagger,i}(\bx)}\,
\bigl(-\frac{\overrightarrow{\dd}}{\dd x^\alpha}\bigr)
\frac{\overrightarrow{\dd}}{\dd u^i_{x^\alpha}} f(\bx,[\bu],[\bxi])\bigr|_{j^\infty_\bx(\bolds,\bolds^\dagger)}\bigr\rangle
\\
{}&\stackrel{\text{def}}{=} \int\Id\bby\ \delta s^i(\bby)
\langle\underrightarrow{\vec{e}_i(\bby),\vec{e}^{\,\dagger,i}(\bx)}\rangle\cdot
\Bigl(\bigl(-\frac{\overrightarrow{\Id}}{\Id x^\alpha}\bigr)
\frac{\overrightarrow{\dd}}{\dd u^i_{x^\alpha}} f(\bx,[\bu],[\bxi])\Bigr)\Bigr|_{j^\infty_\bx(\bolds,\bolds^\dagger)}.
\end{align*}%multline*}
For multi\/-\/indexes~$\sigma$ longer than~$\{x^\alpha\}$ the powers $(\overleftarrow{\dd}/\dd\bby)^\sigma$ are processed by repeated integrations by parts; this yields $(-\overrightarrow{\Id}/\Id\bx)^\sigma$. 
In the course of derivation of densities with respect to not~$a^i_\sigma$ but~$b_{j,\tau}$ and so, in the course of using the other of two (co)\/vectors' couplings, all reasonings are still performed in the exactly same way.
\end{proof}

\begin{cor}
The derivatives w.r.t.\ base variables are transported along an edge to the arrowhead according to the scenarios drawn in Fig.~\ref{FigByParts}; 
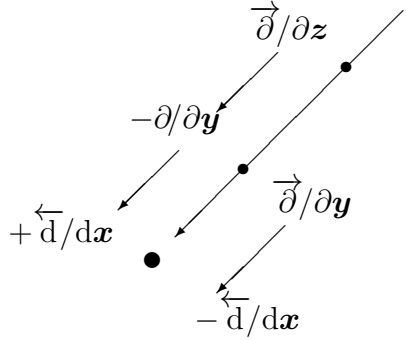
\begin{figure}[htb]
\begin{center}
\unitlength=1.5mm
\linethickness{0.4pt}
\begin{picture}(26.33,26.33)
\put(25.00,25.00){\vector(-1,-1){20.00}}
%\put(23.67,26.33){\line(1,-1){2.67}}
%\put(19.00,21.67){\line(1,-1){2.67}}
%\put(10.67,13.33){\line(1,-1){2.67}}
\put(20,20){\circle*{1}}
\put(11,11){\circle*{1}}
\put(5.33,12.67){\vector(-1,-1){5.33}}
\put(14.67,6.00){\vector(-1,-1){6.00}}
\put(14.00,21.33){\vector(-1,-1){5.33}}
%\put(12.67,22.67){\line(1,-1){2.67}}
%\put(4.00,14.00){\line(1,-1){2.67}}
%\put(13.33,7.33){\line(1,-1){2.67}}
\put(3.00,3.00){\circle*{1.33}}
\put(15.00,23.67){\makebox(0,0)[cc]{$\overrightarrow{\partial}\!/\partial\bz$}}
\put(5.00,15.33){\makebox(0,0)[cc]{$-%\overrightarrow
{\partial}\!/\partial\bby$}}
\put(-5.00,5.67){\makebox(0,0)[cc]{$+\overleftarrow{\mathrm{d}}\!/\mathrm{d}\bbx$}}
\put(11.33,-2.00){\makebox(0,0)[cc]{$-\overleftarrow{\mathrm{d}}\!/\mathrm{d}\bbx$}}
\put(17.00,8.33){\makebox(0,0)[cc]{$\overrightarrow{\partial}\!/\partial\bby$}}
\end{picture}
\vspace{.5cm}
\caption{Each on-the-diagonal push of a derivative along the edge creates an extra minus sign.}\label{FigByParts}
\end{center}\label{FigVariationsArrows}
\end{figure}
each derivative is referred to the copy of base manifold~$M^m$ over which the object or structure it acts on %/ upon
is defined.
\end{cor}

\subsubsection{}%{\S3.2.2.}
\label{pVariationsArrows}
Let us now explain how the edges in Kontsevich's graphs get oriented; %in fact, 
this mechanism is unseparable from the integration by parts.
%, which we re-address in the succeding %/ succedent
%section. 
%Geometrically, 
Every edge is realised by the linking of variations --~with respect to the canonical conjugate 
variables~$u^i$ and~$\xi_i$~-- of objects that are contained in the two adjacent vertices.
The orientation of such edge is the ordering $\delta\bs^{\dagger}\prec\delta\bs$ of singular linear integral operators; initially, they act as shown in Fig.~\ref{FigOrientArrow}.
\begin{figure}[htb]
\begin{center}
\unitlength=1mm
\linethickness{0.4pt}
\begin{picture}(28.67,26.99)
\put(25.00,25.00){\vector(-1,-1){20.00}}
\put(3.33,3.33){\circle*{1.833}}
\put(26.33,26.33){\circle*{1.833}}
%\put(23.2,26.8){\line(1,-1){3.6}}
\put(12.67,16.33){\makebox(0,0)[cc]{$i$}}
\put(22.00,19){\makebox(0,0)[lc]{$\overrightarrow{\delta s}^{\dagger}_i(\bz)$}}
\put(11.00,7){\makebox(0,0)[lc]{$\overleftarrow{\delta s}^i(\bby)$}}
\put(28.67,26.33){\makebox(0,0)[lc]{$\text{Obj}({\text{dvol}(\bbx_1)})$}}
\put(5.00,0){\makebox(0,0)[lc]{$\text{Obj}({\text{dvol}(\bbx_2)})$}}
%\put(9.7,13){\line(1,-1){3.6}}
%\put(15.7,19){\line(1,-1){3.6}}
\put(11.2,11.2){\circle*{1.5}}
\put(20.2,20.2){\circle*{1.5}}
\end{picture}
\caption{Decorated by $i$, such oriented edge appeared in the operation~$\bullet$ in~\cite{KontsevichFormality}; here we extend the formula which the edge encodes to the set\/-\/up of affine bundles over~$M^m$ by letting the points $\bx_1$,\ $\bx_2$,\ $\bby$,\ and~$\bz$ run through the integration domain~$M^m$ of
positive dimension~$m$.}\label{FigOrientArrow}
\end{center}
\end{figure}
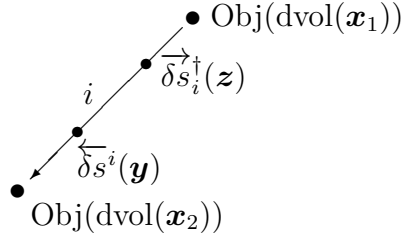
Every edge in an oriented graph~$\Gamma$ contributes to the summand 
(which $\Gamma$~encodes in the star\/-\/product~$\star_{\hbar}$) by the linking of variations and by the linking of differentials of objects contained in the vertices.
Novel with respect to the classical set-up of section~\ref{SecFinite}, the variations~$\delta\bs$ and~$\delta\bs^{\dagger}$ absorb
%are
%introduced / 
%brought into the picture in order to properly handle 
the derivatives 
--~previously, non\/-\/existent~-- along the base manifold~$M^m$.

The singular linear integral operators $\delta\bs$ and $\delta\bs^{\dagger}$ now act not only on the spaces of local functionals but also on elements of their native space of singular distributions. At the same time, regular integral functionals --~like~$\bcP$ and~$F$ or~$G$ from~$\bcA$,
which are contained in vertices of a graph~$\Gamma$ at hand,~-- themselves can discard the copy of volume elements $\dvol(\,\cdot\,)$ which they are equipped with and by this, reshape into singular integral operators. 
%%%
The linking of normalised variations yields %the 
singular linear integral operators that act via 
multiplication by~$\pm1$. %In turn, 
At the end of the day, the linking of every two neighbouring objects converts one of them into a singular linear integral operator such that the (co)\/vectors contained in it act on their duals, resulting in the multiples~$+1$ or~$-1$.

\begin{example}
The edge $\bcP\xrightarrow{\:i\:}%\stackrel{i}{\to} 
F$ encodes the formula${}^{\text{\ref{FootForwardDelay}}}$
\begin{multline}
\iint\Id\bx_1\cdot\Bigl(\tfrac12\xi_{\alpha}P^{\alpha\beta}_{\sigma}\bigl|_{(\bx_1,[\bu])}\xi_{\beta,\sigma}\Bigr)
\frac{\overleftarrow{\dd}}{\dd\xi_{i,\tau}}
\\
\Bigl{\langle}
\underline{\underline{\stackrel{\text{first}}{\vec{e}_i(\bx_1)}}}|\ %
\iint\Id\bby_1\Id\bby_2\,\Bigl{\langle}\underline{\stackrel{\text{first}}{(-\vec{e}^{\mathstrut\ \dagger,i})(\bby_1)}}
\cdot\delta s_i^{\dagger}(\bby_1)|\delta s^i(\bby_2)\cdot
\underline{\stackrel{\text{second}}{\vec{e}_i(\bby_2)}}\Bigr{\rangle}\ %
|\underline{\underline{\stackrel{\text{second}}{\vec{e}^{\mathstrut\ \dagger,i}(\bx_2)}}}\Bigr{\rangle}
\\
\phantom{\Bigl{|}}^{\lceil}
%\Bigl\lceil
\Bigl(+\frac{\overrightarrow{\Id}}{\Id\bx_2}\Bigr)^{\tau}
\Bigl(-\frac{\overrightarrow{\Id}}{\Id\bx_2}\Bigr)^{\sigma_2}
%\Bigr\rceil
\phantom{\Bigr{|}}^{\rceil}
\frac{\overrightarrow{\dd}}{\dd u^i_{\sigma_2}}(f)(\bx_2,[\bu])\cdot\dvol(\bx_2).\label{pExampleGraphFormula}
\end{multline}
The singular distributions wright the diagonal $\bx_1=\bby_1=\bby_2=\bx_2$; both couplings evaluate to~$+1$, and normalisation~\eqref{EqNormalise} makes the edge's cargo invisible (indeed, it contributes via multiplication 
by~$+1$). 
\end{example}
%, hence not immediate / easy
% to recognise / discover.

\begin{rem}
The three singular operators in~\eqref{pExampleGraphFormula} can be directed in the opposite way, 
i.e.\ against the edge orientation along which the derivatives are transported %[anyway / 
in any case. This would keep the volume element at the arrow tail.
\end{rem}

\begin{example}
Consider the edge $\bcP\xrightarrow{\:j\:}%\stackrel{j}{\to} 
G$ encoding the formula${}^{\text{\ref{FootForwardDelay}}}$ %/ yielding the formula]
\begin{multline*}
\iint\dvol(\bx_1)
\Id\bx_2\cdot\Bigl(\tfrac12\xi_{\alpha}P^{\alpha\beta}_{\sigma}\bigl|_{(\bx_1,[\bu])}\xi_{\beta,\sigma}\Bigr)
\frac{\overleftarrow{\dd}}{\dd\xi_{j,\tau}}
\\
\Bigl{\langle}
\underline{\underline{\stackrel{\text{second}}{\vec{e}_j(\bx_1)}}}|\ %
\iint\Id\bby_1\Id\bby_2\,\Bigl{\langle}\underline{\stackrel{\text{second}}{(-\vec{e}^{\mathstrut\ \dagger,j})(\bby_1)}}
\cdot\delta s_j^{\dagger}(\bby_j)|\delta s^j(\bby_2)\cdot
\underline{\stackrel{\text{first}}{\vec{e}_j(\bby_2)}}\Bigr{\rangle}\ %
|\underline{\underline{\stackrel{\text{first}}{\vec{e}^{\mathstrut\ \dagger,j}(\bx_2)}}}\Bigr{\rangle}
\\
\phantom{\Bigl{|}}^{\lceil}
%\Bigl\lceil
\Bigl(+\frac{\overrightarrow{\Id}}{\Id\bx_2}\Bigr)^{\tau}
\Bigl(-\frac{\overrightarrow{\Id}}{\Id\bx_2}\Bigr)^{\sigma_2}
%\Bigr\rceil
\phantom{\Bigr{|}}^{\rceil}
\frac{\overrightarrow{\dd}}{\dd u^j_{\sigma_2}}(g)(\bx_2,[\bu]).
\end{multline*}
In this case, both couplings evaluate to~$-1$ by~\eqref{EqTwoCouplings} still their values' product is~$+1$; the diagonal\/-\/making and normalisation mechanism remain the same as before.
\end{example}

Summarising, the ordering in~\eqref{EqTwoCouplings} is the only mechanism that creates sign factors;
%. Let us repeat %[also] that 
the direction in which the operators act along the edge does not necessarily coincide with that edge's orientation
in a graph~$\Gamma$. The arrow specifies the direction to transport the %partial 
derivatives by using the integrations by parts.

\begin{define}\label{DefVarPBr}
The \textsl{variational Poisson bracket} $\{F,G\}_{\bcP}$ of two integral functionals $F$ and $G$ with respect to a given variational Poisson bi-vector $\bcP$ is the graph%
\begin{equation}\label{EqLambdaGraph}
\raisebox{-12pt}{
\unitlength=0.6mm
\linethickness{0.4pt}
\text{\begin{picture}(20.00,20.67)
\put(0.00,5.00){\line(1,0){20.00}}
\put(10.00,20.00){\vector(-1,-2){7.33}}
\put(10.00,20.00){\vector(1,-2){7.33}}
\put(17.67,5.00){\circle*{1.33}}
\put(2.33,5.00){\circle*{1.33}}
\put(10.00,20.00){\circle*{1.33}}
\put(8.33,18.00){\makebox(0,0)[rb]{$\bcP$}}
\put(2.33,11.00){\makebox(0,0)[cc]{$i$}}
\put(18.20,11.00){\makebox(0,0)[cc]{$j$}}
\put(2.33,3.00){\makebox(0,0)[ct]{$F$}}
\put(17.67,3.00){\makebox(0,0)[ct]{$G$}}
\end{picture}}
}\quad.
\end{equation}
%\smallskip
%\noindent
By using two pairs of normalised variations and by letting the volume element stay in the vertex containing~$G$, we realise the geometry of singular distributions %[contained / 
encoded by this picture via the formula\footnote{\label{FootForwardDelay}%
The notation $\lceil\ldots\rceil$, not yet essential in the one\/-\/time variations here, will be explained in Convention~\ref{RemDelay} on p.~\pageref{RemDelay}.}
%%%
\begin{multline*}
\iiint\Id\bx_1\Id\bx(f|_{(\bx_1,[\bu])})\frac{\overleftarrow{\dd}}{\dd u^i_{\sigma_1}}{\phantom{\Bigl|}}
%\Bigl\lceil
\phantom{\Bigl{|}}^{\lceil}
\Bigl(-\frac{\overleftarrow{\Id}}{\Id\bx_1}\bigr)^{\sigma_1}
\Bigl(+\frac{\overleftarrow{\Id}}{\Id\bx_1}\Bigr)^{\tau_1}{\phantom{\Bigl|}}
\vphantom{\Bigr{|}}^{\rceil}
%\Bigr\rceil
\\
\Bigl{\langle}\underline{\stackrel{\text{first}}{\vec{e}^{\mathstrut\ \dagger,i}(\bx_1)}} |\ %
\iint\Id\bby_1\Id\bby_2\,\Bigl{\langle}
\stackrel{\text{first}}{\vec{e}^i(\bby_2)}\cdot\delta s^i(\bby_2)|\delta s_i^{\dagger}(\bby_1)\cdot
\stackrel{\text{second}}{(-\vec{e}^{\mathstrut\ \dagger,i})}(\bby_1)\Bigr{\rangle}\ %
|\underline{\stackrel{\text{second}}{\vec{e}_i(\bx)}}\Bigr{\rangle}\cdot{}
\\
{}\cdot\frac{\overrightarrow{\dd}}{\dd\xi_{i,\tau_1}}
\Bigl(\tfrac12\xi_{\alpha}P^{\alpha\beta}_{\zeta}\Bigl|_{(\bx,[\bu])}\xi_{\beta,\zeta}\Bigr)
\frac{\overleftarrow{\dd}}{\dd\xi_{j,\tau_2}}\cdot{}
\\
{}\cdot\Bigl{\langle}\underline{\underline{\stackrel{\text{first}}{\vec{e}_j(\bx)}}}|
\ %
\iint\Id\bz_1\Id\bz_2\,\Bigl{\langle}\stackrel{\text{first}}{(-\vec{e}^{\mathstrut\ \dagger,j})(\bz_1)}\cdot
\delta s_j^{\dagger}(\bz_1)|\delta s^j(\bz_2)\cdot
\stackrel{\text{second}}{\vec{e}_j(\bz_2)}\Bigr{\rangle}\ %
|\underline{\underline{\stackrel{\text{second}}{\vec{e}^{\mathstrut\ \dagger,j}(\bx_2)}}}\Bigr{\rangle}
\\
%\Bigl\lceil
\phantom{\Bigl{|}}^{\lceil}
\Bigl(+\frac{\overrightarrow{\Id}}{\Id\bx_2}\Bigr)^{\tau_2}
\Bigl(-\frac{\overrightarrow{\Id}}{\Id\bx_2}\Bigr)^{\sigma_2}
\vphantom{\Bigr{|}}^{\rceil}
%\Bigr\rceil
\frac{\overrightarrow{\dd}}{\dd u^j_{\sigma_2}}(g|_{(\bx_2,[\bu])})
\cdot\dvol(\bx_2).
\end{multline*}
The two pairs of couplings evaluate to $\underline{(-1)}\cdot(-1)\times(+1)\cdot\underline{\underline{(+1)}}=+1$.
The algorithm's output is (cf.~\cite{KuperCotangent,Olver})
%therefore perfectly familiar, for it yields the equality
\begin{equation}\label{EqUsualPBr}
\{F,G\}_{\bcP}%&
=\frac12\int\Bigl{\langle}\frac{\delta F}{\delta\bu}\cdot
\overrightarrow{{A}}\Bigl(\frac{\delta G}{\delta\bu}\Bigr)\Bigr{\rangle}-
\frac12\int\Bigl{\langle}\Bigl(\frac{\delta F}{\delta\bu}\Bigr)
\overleftarrow{{A}}\cdot\frac{\delta G}{\delta\bu}\Bigr{\rangle},%\\
\end{equation}%\intertext{
where $A$ is the Hamiltonian operator built into the variational Poisson bi\/-\/vector $\bcP=\tfrac12\int\langle\bxi\cdot\overrightarrow{{A}}(\bxi)\rangle$. 
\end{define}

\begin{rem}
Because the operator~${A}$ is skew\/-\/adjoint, one could now integrate by parts, obtaining an even shorter expression,
\[%{}&
\cong\int\Bigl{\langle}\frac{\delta F}{\delta\bu}\cdot\overrightarrow{{A}}\Bigl(\frac{\delta G}{\delta\bu}\Bigr)
\Bigr{\rangle}.%\notag
\]
Still let it be remembered that it is the wedge %$\Lambda$-
graph in~\eqref{EqLambdaGraph} that does define the variational Poisson bracket, whereas this %handy short
formula is its remote consequence. 
\end{rem}
%Indeed, much information has been lost in the course of evaluations and, especially, in 
%the course of transporting the --~now, total~-- derivatives to their final positions %/ places
%in the variational derivatives~$\delta/\delta\bu$.

\begin{rem}
When the variational Poisson bracket of two given functionals is assembled by Definition~\ref{DefVarPBr}
--\,to be evaluated at a section~$\bs\in\Gamma(\pi)$,\,-- %of the bundle of physical fields,~-- 
the total derivatives~$\Id/\Id\bx$ immediately follow the partial derivative~ $\dd/\dd u^i_{\sigma}$ in the construction of variational derivatives~$\delta/\delta\bu$. Such inseparability of the horizontal and vertical derivations referring to their own geometries~$M^m$ and~$N^n$, respectively, is specific only to %typical for 
one\/-\/step reasonings (for instance, derivation %production 
of the Euler\/--\/Lagrange equations of motion from a given action functional).
However, a necessity to iterate the virtual shifts of a %the 
section~$\bs\in\Gamma(\pi)$ reveals a %the conceptual 
difficulty of the classical jet bundle geometry (e.\,g., this was ack\-now\-le\-dged 
in~\cite[\S15.1]{HenneauxTeitelboim}).

We shall now explain how the graded permutability of iterated variations is achieved in the course of both all the  couplings evaluation and %especially, in the course of 
transporting the derivatives along~$M^m$ to their final positions.
\end{rem} 

\subsubsection{}%*{\S3.2.3.}
Consider a vertex where two or more arrows arrive~--- or a vertex that contains~$\bcP$ (so that two partial derivatives, $\overrightarrow{\dd}/\dd\xi_{i_1,\tau_1}$ and~$\overleftarrow{\dd}/\dd\xi_{i_2,\tau_2}$, retro\/-\/act on its content) and that serves as the head for another arrow (hence, bringing the partial derivative~$\dd/\dd u^i_{\sigma}$ followed by~$(-\Id/\Id\bx)^{\sigma}$ and possibly, by the total de\-ri\-va\-ti\-ve(s)~$(+\Id/\Id\bx)^{\tau}$ specified by that arrow's tail), see% Fig.
~\eqref{FigDriveThrough}:
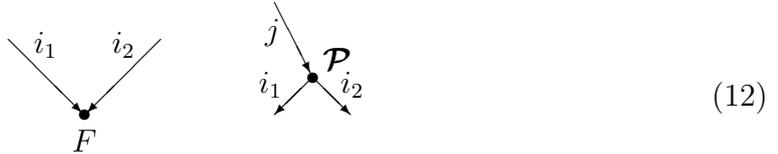
\begin{equation}\label{FigDriveThrough}
%\begin{center}
%\begin{figure}[hb]
\raisebox{-18pt}%
{
\unitlength=1mm
\linethickness{0.4pt}
\begin{picture}(45.00,20)%20.00)
\put(0.00,15.00){\vector(1,-1){9.67}}
\put(20.00,15.00){\vector(-1,-1){9.67}}
\put(10.00,5.00){\circle*{1.33}}
\put(10.00,3.00){\makebox(0,0)[ct]{$F$}}
\put(3.33,12.67){\makebox(0,0)[lb]{$i_1$}}
\put(16.67,12.67){\makebox(0,0)[rb]{$i_2$}}
\put(35.00,20.00){\vector(1,-2){4.67}}
\put(40.00,10.00){\vector(-1,-1){5.00}}
\put(40.00,10.00){\vector(1,-1){5.00}}
\put(40.00,10.00){\circle*{1.33}}
\put(41.33,11.00){\makebox(0,0)[lb]{$\bcP$}}
\put(35.67,16.00){\makebox(0,0)[rc]{$j$}}
\put(36.00,7.33){\makebox(0,0)[rb]{$i_1$}}
\put(43.67,7.33){\makebox(0,0)[lb]{$i_2$}}
\end{picture}
}
%\caption{}\label{FigDriveThrough}
%\end{figure}
%\end{center}
\end{equation}
In which consecutive order are those partial and total derivatives, related to different edges, applied to the content of a vertex\,?

Furthermore, the associativity of Kontsevich star\/-\/product~$\star_\hbar$ is achieved in particular due to many cancellations of similar terms in the associator $(\cdot\mathbin{{\star}_\hbar}\cdot)\mathbin{{\star}_\hbar}\cdot - \cdot\mathbin{{\star}_\hbar}(\cdot\mathbin{{\star}_\hbar}\cdot)$. Consider those three\/-\/sink graphs which can be built using at least two pairs of weighted graphs in the inner and outer products, respectively. Their cancellation prescribes that the resulting analytic expressions, encoded by every such non\/-\/contributing graph with three sinks, must not depend on a scenario to compose that graph. We conclude that the action of total derivatives $\Id/\Id\bx_{\ell}$ in the \textsl{inner} star\/-\/products in the associator $(F\mathbin{\star_{\hbar}}G)\mathbin{\star_{\hbar}}H-F\mathbin{\star_{\hbar}}(G\mathbin{\star_{\hbar}}H)$ is delayed until all the partial derivatives~$\dd/\dd\bu_{\sigma}$ would have finished acting in the \textsl{outer} star\/-\/products.\footnote{This allowed an intrinsic regularisation of the Laplacian~$\Delta$ and variational Schouten bracket~$\lshad\cdot,\cdot\rshad$ in the Batalin\/--\/Vilkovisky formalism (see~\cite{gvbv,sqs13,prg15} and~\cite{CattaneoFelderCMP2000,HenneauxTeitelboim}), the concept was furthered to the formal noncommutative symplectic supergeometry and calculus of cyclic words (see~\cite{cycle16} and~\cite{KontsevichCyclic}).} 

\begin{state}[see~\cite{gvbv,sqs13,prg15,cycle16}]%{convention}
The vertical derivations~$\dd/\dd u^i_{\sigma}$ and (the lifts~$\Id/\Id\bx_{\ell}$ of) horizontal derivations~$\dd/\dd \bx_{\ell}$ %[go / 
are performed %/ occur]
at different stages. %[in a calculation].
First, the vertical derivations~$\dd/\dd\bu_{\sigma}$ along~$N^n$,
together with their counterparts~$\dd/\dd\xi_{\tau}$ from %/ along
the parity\/-\/odd symplectic dual, frame the edges of entire graph~$\Gamma$.
%[At that moment / stage, 
In the %meanwhile / 
meantime, the derivatives along the base~$M^m$ are stored inside the variations~$\delta\bs$ by using~$\dd/\dd\bby_k$.
Lastly, %/ Then,]
all the horizontal derivatives~$(\pm\dd/\dd\bby_k)^{\sigma}$ are channelled from~$\delta s^i$ to~$(\mp\dd/\dd\bx_{\ell})^{\sigma}$, at the end of the day %finally 
acting on the objects which are targets of~$\dd/\dd u^i_{\sigma}$.
\end{state}%{convention}

\begin{convention}\label{RemDelay}
To indicate the delayed arrival of total derivatives to their final places (where we write them at once), let us embrace these operators by using $\lceil\ldots\rceil$ in all formulas (e.\,g., see Definition~\ref{DefVarPBr} on p.~\pageref{DefVarPBr} above).
\end{convention}

\begin{rem}
We emphasize that by the definition of total derivative (see footnote~\ref{FootTD} on p.~\pageref{FootTD}), the derivatives~$\dd/\dd\bby_k(\delta s^i)$ of virtual shifts~$\delta s^i$ for sections~$u^i=s^i(\bx_{\ell})$ reshape, under integration by parts, into the derivatives~$-\dd/\dd\bx_{\ell}(s^i)$ of those sections but they do not affect the %synthetic, 
parity\/-\/odd variables~$\xi_{j,\zeta}$ which parametrise %constitute 
the fibres of another bundle. 
Consequently, %$\Leftarrow$ %Object = P. bracket, not $\bp$
the \textsl{total} derivatives $(+\overrightarrow{\Id}/\Id\bx_{\ell})^{\tau}\circ(-\overrightarrow{\Id}/\Id\bx_{\ell})^{\sigma}$ refer only to the jet space~$J^{\infty}(\pi)$ where %and hence 
they act %only 
on the respective fibre variables~$\bu_{\sigma}$ in a
%which encode the
vertex' content.
\end{rem}

\begin{example}
The first graph in% Fig.
~\eqref{FigDriveThrough} corresponds to the formula
\begin{equation}\label{EqPartialTotal}
\bigl(f|_{(\bx_1,[\bu])}\bigr)\frac{\overleftarrow{\dd}}{\dd u^{i_1}_{\sigma_1}}
\frac{\overleftarrow{\dd}}{\dd u^{i_2}_{\sigma_2}}{\phantom{\Bigl|}}^{\lceil}
\Bigl(-\frac{\overleftarrow{\Id}}{\Id\bx_1}\Bigr)^{\sigma_1\cup\sigma_2}\circ
\Bigl(+\frac{\overleftarrow{\Id}}{\Id\bx_1}\Bigr)^{\tau_1\cup\tau_2}{\phantom{\Bigl|}}^{\rceil}\ ,
\end{equation}
where the multi\/-\/indexes~$\tau_1$ and~$\tau_2$ arrive from the respective arrow tails.

The second graph in Fig.~\ref{FigDriveThrough} contributes with the expression
\begin{multline*}
\phantom{\Biggl|}^{\lceil}
\Bigl(+\frac{\overleftarrow{\Id}}{\Id\bx_k}\Bigr)^{\tau_1}
\phantom{\Biggl|}^{\rceil}
\frac{\overrightarrow{\dd}}{\dd\xi_{i_1,\tau_1}}
\int\tfrac12\xi_{\alpha}\Biggl\{
\phantom{\Biggl|}^{\lceil}\Bigl(+\frac{\overrightarrow{\Id}}{\Id\bx}\Bigr)^{\tau}
\Bigl(-\frac{\overrightarrow{\Id}}{\Id\bx}\Bigr)^{\sigma}
\phantom{\Biggl|}^{\rceil}
\frac{\overrightarrow{\dd}}{\dd u^j_{\sigma}}
\bigl(P^{\alpha\beta}_{\zeta}{\bigr|}_{(\bx,[\bu])}\bigr)\Biggr\}
\,\xi_{\beta,\zeta}\dvol(\bx)\\
\frac{\overleftarrow{\dd}}{\dd\xi_{i_2,\tau_2}}
\phantom{\Biggl|}^{\lceil}
\Bigl(+\frac{\overrightarrow{\Id}}{\Id\bx_{\ell}}\Bigr)^{\tau_2}
\phantom{\Biggl|}^{\rceil}\ ,
\end{multline*} 
where the multi\/-\/index~$\tau$ arrives from the tail of arrow decorated with~$j$ and where the copies of base~$M^m$ for objects at the heads of arrows that carry~$i_1$ and~$i_2$ are marked using~$k$ and~$\ell$, respectively.\\[2pt]
\centerline{\rule{1in}{0.7pt}}
\end{example}

\noindent%
Summarising, the local portrait of oriented edges around every vertex in a given graph~$\Gamma$ determines the vertex-incoming partial derivatives with respect to variables~$\bu_{\sigma}$, in-coming graded 
%(by using the ordering relation Left${}\prec{}$Right) 
partial derivatives with respect to the parity-odd variables~$\bxi_{\tau}$, and (powers of) 
delayed $(\pm1)\times$~total derivatives. 
All these derivations act on the object contained in a vertex at hand, that is, on either a Hamiltonian density or structure constants~$P^{\alpha\beta}_{\zeta}\bigl(\bx,[\bu]\bigr)$ of the variational Poisson bracket.
Note that in both cases, the arguments are referred %only 
to the geometry of $J^{\infty}(\pi)$, hence those objects are
expressed %completely determined %entirely 
in terms of %the geometry of physical fields. 
sections of the parity\/-\/even jet bundle~$\pi_\infty\colon J^\infty(\pi)\to M^m$.

Globally, each oriented graph~$\Gamma$ in the Kontsevich summation formula~\eqref{EqUniversalStar} 
encodes %determines 
a singular linear integral operator 
that acts on a local functional %~--- which is 
contained in one of the~sinks.
%oriented graph's terminal vertices (or sinks). It is 
%obvious that there remain no letters $\bxi$, neither in such operators themselves nor in the objects they produce.

\subsection{%*{\S3.3. 
%The sought\/-\/for associativity of~$\star_{\hbar}$. Why it leaks
On the associativity of star\/-\/product%leak estimate $\Assoc_{\star_\hbar}(F,G,H)\doteq\bar{o}(\hbar^{\geqslant 3})$
}\label{SecJacobi}

\subsubsection{}\label{SecMoyal}%*{\S3.3.4.}
Let $\bcP=\tfrac12\int\langle\bxi\cdot\left.P_{\tau}\right|_{\bx}(\bxi)\rangle$ be a variational Poisson bi\/-\/vector such that its coefficients~$P^{ij}_{\tau}(\bx)$ do not depend explicitly on the jet fibre variables~$\bu_{\sigma}$ in the bundle~$\pi_\infty$ over~$M^m\ni\bx$. Let $F=\int f(\bx_1,[\bu])\cdot\dvol(\bx_1)$ and
$G=\int g(\bx_2,[\bv])\cdot\dvol(\bx_2)$ be integral functionals referred to two identical copies of the
jet space~$J^{\infty}(\pi)$. A variational generalisation~$F\star G$ of the Moyal\/--\/Gr\"onewold\/--\/Weyl star\/-\/product~$\star$ for~$F$ and~$G$ is the local functional %--~that is, itself a scalar~-- 
which is constructed from~\eqref{EqMoyal} for the variational Poisson bracket $\{\cdot,\cdot\}_{\bcP}$ by using the techniques from~\S\ref{SecElements}. %geometry of iterated variations 
The variational Moyal product of two integral functionals is expressed by the formula
\begin{multline}\label{EqVarMoyal}
F\star G = %\cong
\int\Biggl(\left.f\right|_{(\bx_1,[\bu])}
\exp\biggl(\frac{\overleftarrow{\dd}}{\dd u^i_{\sigma}}
{\vphantom{\Bigl(}}^{\lceil}\Bigl(-\frac{\overleftarrow{\Id}}{\Id\bx_1}\Bigr)^{\sigma}
\Bigl(\frac{\overleftarrow{\Id}}{\Id\bx_1}\Bigr)^{\tau\ \rceil}\cdot%\\
%
%\cdot
\frac{\overrightarrow{\dd}}{\dd\xi_{i,\tau}}\Bigl(\frac{\hbar}2\xi_{\alpha}P^{\alpha\beta}_{\lambda}(\bx)
\xi_{\beta,\lambda}\Bigr)\frac{\overleftarrow{\dd}}{\dd\xi_{j,\zeta}}\cdot{}\\
{}\cdot
{\vphantom{\Bigl(}}^{\lceil}\Bigl(\frac{\overrightarrow{\Id}}{\Id\bx_2}\Bigr)^{\zeta}
\Bigl(-\frac{\overrightarrow{\Id}}{\Id\bx_2}\Bigr)^{\chi\ \rceil}
\frac{\overrightarrow{\dd}}{\dd v^j_{\chi}}\biggr)
g{\bigr|}_{(\bx_2,[\bv])}\Biggr)\Biggr|_{\begin{smallmatrix}\bx_1=\bx=\bx_2\\
\left[\bu\right]=\left[\bv\right]
\end{smallmatrix}}
%{\substack{\bx_1=\bx=\bx_2\\[\bu]=[\bv]}}
\cdot\dvol(\bx).
\end{multline}
The angular brackets $\lceil\ldots\rceil$ in~\eqref{EqVarMoyal} embrace the total derivatives whose action  --\,in every term of the towered wedge graph expansion of~$\star$~-- antecedes\footnote{\label{FootDelayInMoyal}%
We recall that %Likewise, 
the action of total derivatives contained, e.\,g., 
in $F\mathbin{\star}G$ itself constituting a part of the object 
$(F\mathbin{\star}G)\mathbin{\star}H - F\mathbin{\star}(G\mathbin{\star}H)$
is also delayed until all the partial derivatives would have acted on the densities~$f$,\ $g$, or~$h$.}
%%%
the action of partial derivatives with respect to $u^i_{\sigma}$ and~$v^j_{\chi}$.
For instance, the expansion starts as follows:
\begin{multline*}
F\mathbin{\star}G = F\times G
+\frac{\hbar^1}{1!}\{F,G\}_{\bcP}+\frac{\hbar^2}{2!}\int\left(f\bigl|_{(\bx,[\bu])}\right)
\frac{\overleftarrow{\dd}}{\dd u^{i_1}_{\sigma_1}}\frac{\overleftarrow{\dd}}{\dd u^{i_2}_{\sigma_2}}
\left(-\frac{\overleftarrow{\Id}}{\Id\bx}\right)^{\sigma_1\cup\sigma_2}
\left(\frac{\overleftarrow{\Id}}{\Id\bx}\right)^{\tau_1\cup\tau_2}\cdot\\
\cdot\frac{\overrightarrow{\dd}}{\dd\xi_{i_1,\tau_1}}
\left(\tfrac12\xi_{\alpha_1}P^{\alpha_1\beta_1}_{\lambda_1}(\bx)\xi_{\beta_1,\lambda_1}\right)
\frac{\overleftarrow{\dd}}{\dd\xi_{j_1,\zeta_1}}
\cdot\frac{\overrightarrow{\dd}}{\dd\xi_{i_2,\tau_2}}
\left(\tfrac12\xi_{\alpha_2}P^{\alpha_2\beta_2}_{\lambda_2}(\bx)\xi_{\beta_2,\lambda_2}\right)
\frac{\overleftarrow{\dd}}{\dd\xi_{j_2,\zeta_2}}\cdot\\
\cdot\left(\frac{\overrightarrow{\Id}}{\Id\bx}\right)^{\zeta_1\cup\zeta_2}
\left(-\frac{\overrightarrow{\Id}}{\Id\bx}\right)^{\chi_1\cup\chi_2}
\frac{\overrightarrow{\dd}}{\dd u^{j_2}_{\chi_2}}\frac{\overrightarrow{\dd}}{\dd u^{j_1}_{\chi_1}}
\left(g\bigl|_{(\bx,[\bu])}\right)\cdot\dvol(\bx)+\ov{o}(\hbar^2).
\end{multline*}
%Produced from the variational Poisson structure $\{\cdot,\cdot\}_{\bcP}$, 
Let $\widetilde{\bx}=I_{\alpha\beta}\,\bx+\vec{\boldsymbol{\mu}}$, $\widetilde{\bu}=J_{\alpha\beta}\,\bu+\vec{\boldsymbol{\nu}}(\bx)$ be an affine change of variables in~$\pi$ such that the Jacobian matrix~$J$ is locally constant on the intersection~$V_\alpha\cap V_\beta$ of two charts~$V_\alpha$,\ $V_\beta\subseteq M^m$.
Then formula~\eqref{EqVarMoyal} is invariant with respect to such coordinate reparametrisation.
%%%
%In any other system of coordinates on $J^{\infty}(\pi)$ the star\/-\/product $F\mathbin{\star_{\hbar}}G$ is expressed by using the postulate %of its covariance. 
%that this local functional is a scalar indeed: the total derivatives $\Id/\Id\bx$ and (the derivatives with respect to) jet fibre coordinates~$\bu_{\sigma}$ are replaced --~via the chain rule~-- just where they are in the above formula.
%%%
The associativity of~\eqref{EqVarMoyal} is proved in a standard way 
(see the proof of Proposition~\ref{PropMoyalAssoc} on p.~\pageref{PropMoyalAssoc}).
 %and footnote~\ref{FootDelayInMoyal} on %the preceding page).
 %p.~\pageref{FootDelayInMoyal}).
The associator $(F\star G)\star H - F\star(G\star H)$ of three given integral functionals over $J^{\infty}(\pi)$ itself is an \textsl{integral} %({sic}!)
functional whose density is identically zero at all points %$(\bx,[\bu])$ 
of~$J^{\infty}(\pi)$ over~$M^m$.

\begin{state}
Formula~\eqref{EqVarMoyal} provides the deformation quantisation of first and, via factorisation by using the junior Poisson bracket for the modified system, of second variational Poisson structures for the Drinfel'd\/--\/Sokolov hierarchies. 
\end{state}
%(e.\,g., for the Korteweg\/--\/de~Vries equation that corresponds to the root system $\mathsf{A}_1$ of Lie algebra $\Sl_2(\BBC)$ --- or to the Virasoro algebra).

\begin{example}[root system~$\mathsf{A}_1$]
Consider the Korteweg\/--\/de~Vries equation
\[
w_t=-\tfrac12w_{xxx}+6ww_x=
\left(-\tfrac12 D^3_x+2wD_x+2D_x\circ w\right)\left(\frac{\delta}{\delta w}\int\tfrac12 w^2\Id x\right)
\]
realised by using its second, field\/-\/dependent variational Poisson structure.\footnote{Through the Fourier transform, the Hamiltonian operator~$\hat{A}_2^{\text{KdV}}$ encodes the Virasoro algebra, cf.~\cite{BelavinW}.}
%%%
Consider the Miura substitution~\cite{Miura68} $w=\frac12(u_x^2-u_{xx})$; let us explain in advance that the conserved current $w\,\Id x=\frac12(u_x^2-u_{xx})\,\Id x$ stems --\,via the First Noether theorem\,-- from the Noether symmetry $\varphi_1=u_x$ of the action $\cL=\iint\left(\frac12u_xu_y+\frac12e^{2u}\right)\,\Id x\wedge\Id y$ for the Liou\-vil\-le equation $\cE_{\text{Liou}}=\{u_{xy}=\exp(2u)\}$.

The mapping $w=w\bigl(x,[u]\bigr)$ is determined by the \textsl{integral} 
$w\in\ker\left.\frac{\Id}{\Id y}\right|_{\cE_{\Liou}}$; we recall that the coefficient ``2" in the right-hand side of
$u_{xy}=\exp(2u)$ is the only entry of the Cartan matrix $K=\|2\|^1_1$ for Lie algebra $\Sl_2(\BBC)$.

By definition, put $\vartheta=\frac12u_x$ so that $\ell^{(u)}_{\vartheta}=\frac12\frac{\Id}{\Id x}$ is the first
Hamiltonian operator $\hat{B}_1^{\mKdV}$ of modified KdV hierarchy and so that $w=2\vartheta^2-\vartheta_x$.
Denote by $\square=4\vartheta+D_x=2u_x+D_x$ the adjoint $(\ell_w^{(\vartheta)})^{\dagger}$ of linearisation
$\ell_w^{(\vartheta)}=4\vartheta-D_x$. By using the chain rules
$\delta/\delta\vartheta=(\ell_w^{(\vartheta)})^{\dagger}\circ\delta/\delta w$ and
$\delta/\delta u=(\ell_{\vartheta}^{(u)})^{\dagger}\circ(\ell_w^{(\vartheta)})^{\dagger}\circ\delta/\delta w$,
we cast the (potential) modified KdV equations,
\begin{align*}
u_t&=-\tfrac12u_{xxx}+u^3_x=\square(w),\qquad
\vartheta_t=-\tfrac12\vartheta_{xxx}+12\vartheta^2\vartheta_x,\\
\intertext{into their canonical De Donder\/--\/Weyl's 
representation~\cite{DeDonderWeyl}}
u_t&=\frac{\delta H[w[\vartheta]]}{\delta\vartheta},\qquad
\vartheta_t=-\frac{\delta H[w[\vartheta[u]]]}{\delta u}
\quad\text{ with }H=\int\tfrac12w^2\,\Id x.\\
\intertext{Clearly, we then recover the KdV evolution}
w_t&=\bigl(\ell_w^{(\vartheta)}(\vartheta_t)\bigr)[w]=
\left(-\tfrac12 D_x^3+4wD_x+2w_x\right)\left(%\frac
{\delta H\bigl(x,[w]\bigr)}/{\delta w}\right).
\end{align*}
This factorisation pattern involving the Fr\'echet derivatives (or \textsl{linearisations}),
\[
\hat{A}_2^{\KdV}=\ell_w^{(\vartheta)}\circ\ell_{\vartheta}^{(u)}\circ\left(\ell_w^{(\vartheta)}\right)^{\dagger},
\]
is common %for / 
to all the root systems of ranks~$r\geqslant1$, that is, for the (modified) Drinfel'd\/--\/Sokolov 
hierarchies~\cite{DSViniti84}.
It is seen that the hierarchy for respective analogue of potential modified KdV equation for $\bu$ constitutes the maximal
commutative subalgebra in the Lie algebra of Noether symmetries for Leznov\/--\/Saveliev's nonperiodic 2D~Toda chains~\cite{LeznovSaveliev} $u_{xy}^i=\exp\Bigl(\sum\limits_{j=1}^r\frac{2\langle\alpha_i,\alpha_j\rangle}{\langle\alpha_j,\alpha_j\rangle}\cdot u^j\Bigr)$. The algorithm for construction of $r$~integrals~\cite{ZhiberSokolov} $w^1$,\ $\dots$,\ $w^r$ is known from~\cite{Shabat95}, see~\cite{Protaras2008} for an illustration. 
The De~Donder\/--\/Weyl formalism~\cite{DeDonderWeyl} furthers the approach:
the variables $\vartheta_1$,\ $\ldots$,\ $\vartheta_r$ are the canonical conjugate \textsl{momenta}, $\vartheta_i=\dd L/\dd u^i_y$,
for the genuine \textsl{coordinates} $u^1$,\ $\ldots$,\ $u^r$ satisfying the 2D~Toda equations.
The Lagrangian density is $L=\tfrac12\kappa_{ij}u^i_xu^j_y+\langle a_i,\exp(K^i_{\mathstrut\,j}u^j)\rangle$,
where each row of the Cartan matrix $K=\|K^i_{\mathstrut\,j}\|$ is symmetrised to
$\kappa=\|a_i\cdot K^i_{\mathstrut\,j}\|_{i=1,\ldots,r}^{j=1,\ldots,r}$
by using the root lengths, $a_i\mathrel{{:}{=}}1/\langle\alpha_i,\alpha_i\rangle$ at every~$i$. Consequently, the junior variational
Poisson structure for the modified Drinfel'd\/--\/Sokolov hierarchy is
\[
\widehat{B}_1=\left\|\frac{\langle\alpha_i,\alpha_j\rangle}{\langle\alpha_i,\alpha_i\rangle\langle\alpha_j,\alpha_j\rangle}\,
\frac{\Id}{\Id x}\right\|_{i=1,\ldots,r}^{j=1,\ldots,r}
\]
for every root system~$\alpha_1,\ldots,\alpha_r$. Its coefficients are constants~$\in\Bbbk$.
\end{example}

Having thus factorised a higher, field\/-\/dependent variational Poisson structure through the junior variational Poisson structures whose coefficients %of which 
do not depend explicitly on the new fields~$\phi\in\Gamma(\pi)$, 
we reduce the large deformation quantisation problem for functionals 
$F[\bw]$,\ $G[\bw]$,\ $H[\bw]\colon\Gamma(\wt{\pi})\to\Bbbk$
to a much smaller %computationally much simpler 
Moyal\/--\/Gr\"onewold\/--\/Weyl case~\eqref{EqVarMoyal} of the same functionals 
$F\bigl[\bw[\bu]\bigr]$,\ $G\bigl[\bw[\bu]\bigr]$,\ $H\bigl[\bw[\bu]\bigr]\colon\Gamma(\pi)\to\Bbbk$, 
now referred to the jet bundle~$\pi_\infty$ of (potential) modified hierarchies.

Indeed, let~$\pi$ and~$\wt{\pi}$ be two affine bundles over the base~$M^m$.
Consider a jet space morphism $\bw^{(\infty)}\colon$ $J^\infty(\pi)\to J^\infty(\wt{\pi})$
specified by a Miura\/-\/type substitution $\bw=\bw(\bx,[\bu])\colon\Gamma(\pi_{(\infty)})\to\Gamma(\wt{\pi})$ of positive differential order. %Arguing as above, We see that 
The Hamiltonian differential operators factorise via\footnote{Integrating by parts, 
\begin{multline*}
\tfrac12\int
\Biggl{\langle}(\bxi)
 \Bigl(\overleftarrow{\ell}_{\bw}^{(\bu)}\Bigr)^{\dagger}\cdot
 \Bigl(P_{\tau}{\bigr|}_{(\bx,[\bu])}
   \Bigl(\frac{\Id}{\Id\bx}\Bigr)^{\tau}\circ
%   \Bigl(
\overrightarrow{\ell}_{\bw}^{(\bu)\,\dagger}%\Bigr)^{\dagger}
 \Bigr)(\bxi)
\Biggr{\rangle}\cong{}\\
{}\cong\tfrac12\int
\Biggl{\langle}\bxi\cdot\Bigl(
\overrightarrow{\ell}_{\bw}^{(\bu)}\circ P_{\tau}{\bigr|}_{(\bx,[\bu])}
\Bigl(\frac{\Id}{\Id\bx}\Bigr)^{\tau} \circ %\Bigl(
\overrightarrow{\ell}_{\bw}^{(\bu)\,\dagger}%\Bigr)^{\dagger}
\Bigr)(\bxi)\Biggr{\rangle}
=\tfrac12\int\langle\bxi,\overrightarrow{A}(\bxi)\rangle\ ,
\end{multline*}
we construct the Hamiltonian differential operator in total derivatives that takes variational covectors to (the generating sections of) evolutionary vector fields.}
%%%
\begin{equation}\label{EqQuattro}
A{\bigl|}_{(\bx,[\bw])}=
\overrightarrow{\ell}_{\bw}^{(\bu)}\circ B{\bigr|}_{(\bx,[\bu])}\circ
%\Bigl(
\overrightarrow{\ell}_{\bw}^{(\bu)\,\dagger}%\Bigr)^{\dagger}
,
\end{equation}
where $\{\cdot,\cdot\}_{\frac12\int\langle\boldsymbol{\chi},A(\boldsymbol{\chi})\rangle}$ is the variational Poisson bracket \textsl{induced}\ for functionals
$H[\bw]\colon$ $\Gamma(\wt{\pi})\to\Bbbk$ from a \textsl{given} variational Poisson structure
$\{\cdot,\cdot\}_{\frac12\int\langle\bxi\cdot B(\bxi)\rangle}$ for the 
pull\/-\/backs ${H\bigl[\bw[\bu]\bigr]\colon\Gamma(\pi)\to\Bbbk}$.

Formula~\eqref{EqQuattro} correlates senior Poisson structures
$\{\cdot,\cdot\}_{\frac12\int\langle\boldsymbol{\chi},A_{i+1}(\boldsymbol{\chi})\rangle}$
for multi\/-\/Ha\-mil\-to\-ni\-an hierarchies with the junior Hamiltonian 
operators~$B_i$ for the respective modified hierarchies of 
completely integrable PDE systems (see~\cite{TMF2004,TMF2009} and references therein).
Solutions $\bigl(\pi$,\ $\bw\bigl(\bx,[\bu]\bigr)$,\ $B\bigr)$ are ``good" if the coefficients of differential operator~$B$, referred to~$(\bx,u^j_{\sigma})$,
do not contain the jet variables $u^j_{\sigma}$ explicitly 
(that is, %$B$~is Darboux\/-\/canonical in the sense of Corollary~\ref{CorAntiDarboux}
the star\/-\/product for the Hamiltonian operator~$B$ amounts to formula~\eqref{EqVarMoyal}). 
In this case the output of deformation quantisation procedure~$\times\mapsto\star$ is associative at all orders of the deformation parameter~$\hbar$. However, for a given Hamiltonian differential operator~$A$ over~$J^{\infty}(\wt{\pi})$, its factorisation problem can be very hard.

\subsubsection{}\label{SecWhyLeaks}%*{\S3.3.5.}
%\smallskip
Finally, let us take a generic variational Poisson brackets $\{\cdot,\cdot\}_{\bcP}$ with field\/-\/dependent coefficients~$P^{ij}_{\tau}(\bx,[\bu])$. For instance, suppose that 
factorisation~\eqref{EqQuattro}, reducing a given Hamiltonian differential operator~$\widehat{A}_2$ on~$J^{\infty}(\wt{\pi})$ to the Moyal case for~$\widehat{B}_1$ on~$J^{\infty}(\pi)$, is not yet known.

It is readily seen that the splitting of differential consequences from the Jacobi identity into the separately vanishing homogeneous components, see Proposition~\ref{PropLeibnizGraphZero} in~\S\ref{SecJacVanishVia}, no longer takes place without reservations in the variational setting. This is because not only the vertical derivatives along the fibre of~$J^\infty(\pi)$ work by the Leibniz rule over the five vertices in every Jacobiator~\eqref{EqJacFig} but also do the total derivatives in their trail, as in~\eqref{EqPartialTotal}, %bind
convert the Jacobiator on its three arguments into an indivisible object. Therefore, in the variational picture only those Kontsevich graph expansions of Leibniz graphs can vanish in which the Jacobiator is not split. Yet we remember from~\cite{sqs15} that this is already not the case at~$\hbar^3$ in the associator for~$\star_\hbar$, see~\S\ref{SecJacVanishVia} above. 

Secondly, for every triple of arguments, the Jacobiator $\Jac_{\bcP}(\cdot,\cdot,\cdot)\cong 0\colon \Gamma(\pi)\to0\in\Bbbk$ is the map which, in terms of~\cite{gvbv,prg15,cycle16}, can be a \textsl{synonym of zero}. Namely, if the density of this cohomologically trivial integral functional is not vanishing over all points of~$M^m$, the local variational polydifferential operator~$\Diamond$ in
\begin{equation}\tag{\ref{EqDefDiamond}}
\Assoc_{\star_\hbar}(F,G,H)=\Diamond\,\bigl(\bcP,\Jac_{\bcP}(\cdot,\cdot,\cdot)\bigr)(F,G,H),
\qquad F,G,H\in\bcA[[\hbar]],
\end{equation}
can produce a nonzero integral functional from its zero\/-\/value argument~$\Jac(\bcP)$. 
Indeed, whenever two or more arrows arrive at a vertex in the argument of~$\Diamond$, 
see~\eqref{FigDriveThrough}, the order in which partial and then total derivatives act is~\eqref{EqPartialTotal}.
Therefore, the mechanism $\delta/\delta\bu\circ\Id/\Id\bx\equiv0$ that guarantees the vanishing of the first variation for a cohomologically trivial argument is stepped over. %bypassed.

\begin{cor}
In the field\/-\/theoretic setting, the associativity of star\/-\/product~$\star_\hbar$ can start leaking at order~$\hbar^{\geqslant3}$ for a %generic 
variational Poisson structure~$\{\cdot,\cdot\}_{\bcP}$ with field\/-\/dependent coefficients in the leading deformation term, so that $\Assoc_{\star_\hbar}(\cdot,\cdot,\cdot)\doteq\bar{o}(\hbar^{\geqslant 2})$.
\end{cor}

%In retrospect, our argument reveals why the Virasoro and other $W$-\/algebra 
%struc\-tu\-res$\smash{{}^{\text{\cite{BelavinW,%1989,
%ReshetikhFrenkelSTSh1998%ReshetikhinSTSh98
%}}}}$ do so often arise from the deformation markers~$\{\cdot,\cdot\}_{\bcP}$ 
%--\,via the Fourier transform, then paving a way to Yangians$\smash{{}^{\text{\cite{Brundan,Molev}}}}$\,-- in the various schemes for quantisation of field models.
%We conclude that the existing instruments for calculation of variational Poisson structures do in fact specify points in the moduli spaces of deformation quantisations for field theory models.\\%[2pt]

\subsubsection*{Acknowledgements}
%A.\,K. thanks M.\,Kon\-tse\-vich for posing the problem and many helpful discussions.
The author is grateful to %thanks 
M.\,Kon\-tse\-vich for advice and R.\,Buring for discussion.
The author also thanks the organizers of 50th~Sophus Lie Seminar (26--30~September 2016 in Bedlewo, Poland)
for hospitality. % and many helpful discussions.
This research was supported in part by JBI~RUG project~106552 (Groningen%, The Netherlands
) and by the $\smash{\text{IH\'ES}}$ and MPIM (Bonn), to which the author is grateful for their warm hospitality.

\end{document}